\documentclass[a4paper,11pt]{amsart}
\usepackage{hyperref}
\usepackage{geometry}
\usepackage{comment}

\usepackage[british]{babel}
\usepackage{wrapfig}
\usepackage[comma, square, numbers]{natbib}
\usepackage{yfonts}
\usepackage[utf8]{inputenc}
\usepackage{amssymb}
\usepackage{amsthm}
\usepackage{graphics}
\usepackage{soul}
\usepackage{amsmath}
\usepackage{amstext}
\usepackage{engrec}
\usepackage{floatflt}
\usepackage{rotating}
\usepackage[safe,extra]{tipa}
\usepackage[font={footnotesize}]{caption}
\usepackage{framed}
\usepackage{listings}
\usepackage{subcaption}
\usepackage{enumerate}
\usepackage{tikz}

\usepackage{mathtools}

\usepackage[titletoc,title]{appendix}

\newcommand{\vertiii}[1]{{\left\vert\kern-0.25ex\left\vert\kern-0.25ex\left\vert #1 
    \right\vert\kern-0.25ex\right\vert\kern-0.25ex\right\vert}}

\newcommand{\mc}{\mathcal}
\newcommand{\mb}{\mathbb}

\newcommand{\R}{\mb R}

\newcommand{\N}{\mb N}

\newcommand{\T}{\mb T}

\newcommand{\eea}{\end{align}}

\renewcommand{\epsilon}{\varepsilon}
\renewcommand{\bar}{\overline}
\renewcommand{\tilde}{\widetilde}

\newcommand{\bo}{\boldsymbol}
\renewcommand{\phi}{\varphi}

\DeclareMathOperator{\supp}{supp}

\renewcommand\upsilon{\theta}

\usepackage[normalem]{ulem}

\newtheorem{theorem}{Theorem}[section]

\newtheorem{corollary}{Corollary}[section]
\newtheorem{lemma}{Lemma}[section]

\newtheorem{proposition}{Proposition}[section]
\theoremstyle{definition}
\newtheorem{definition}{Definition}[section]

\theoremstyle{remark}
\newtheorem{remark}{Remark}[section]
\newtheorem{example}{Example}[section]

\newtheoremstyle{algorithm}
{4pt}
{4pt}
{}
{}
{}
{:}
{\newline}
{}

\usepackage{xcolor,placeins}

\newtheorem{algorithm}{Algorithm}

\newcommand{\balgorithm}{\begin{algorithm}\begin{framed}\ }
\newcommand{\ealgorithm}{\end{framed}\end{algorithm}}

\newcommand{\pippo}{\begin{code} \  \begin{lstlisting}[frame=single]}
\newcommand{\pappo}{\end{lstlisting} \end{code}}

\newcommand{\bd}{\begin{definition}}
\newcommand{\ed}{\end{definition}}

\newcommand{\bt}{\begin{theorem}}
\newcommand{\et}{\end{theorem}}
\newcommand{\bp}{\begin{proposition}}
\newcommand{\ep}{\end{proposition}}
\newcommand{\bc}{\begin{corollary}}
\newcommand{\ec}{\end{corollary}} 
\newcommand{\bl}{\begin{lemma}}
\newcommand{\el}{\end{lemma}}
\newcommand{\br}{\begin{remark}}
\newcommand{\er}{\end{remark}}

\DeclareMathOperator{\Lip}{Lip}

\theoremstyle{definition}

\theoremstyle{remark}

\theoremstyle{example}

\title[Synchronization of coupled maps]{Synchronization for  networks of globally coupled maps in the thermodynamic limit}
\author{Fanni M. S\'elley}
\address{Institute of Mathematics, University of Leiden, Niels Bohrweg 1, 2333 CA Leiden, The Netherlands}
\email{
	f.m.selley@math.leidenuniv.nl}
\author{Matteo Tanzi}
\address{Courant Institute of Mathematical Sciences, New York University, 251 Mercer St  801, New York, NY 10012}
\email{
	matteo.tanzi@nyu.edu}

\begin{document}
\maketitle

\begin{abstract}{
We study a network of finitely many interacting clusters where each cluster is a collection of globally coupled circle maps in the thermodynamic (or mean field) limit. The state of each cluster is described by a probability measure, and its evolution is given by a self-consistent transfer operator. A cluster is synchronized if its state is a Dirac measure. We provide  sufficient conditions for all clusters to synchronize and we describe setups where the conditions are met thanks to the uncoupled dynamics and/or the (diffusive) nature of the coupling. We also give sufficient conditions for partially synchronized states to arise -- i.e. states where only a subset of the clusters is synchronized -- due to the forcing of a group of cluster on the rest of the network. Lastly, we use this framework to  show emergence and stability of chimera states for these systems. }
\end{abstract}
\tableofcontents

\section{Introduction}

Coupled map systems are simple models of spatially ordered interacting units, also referred to as \emph{sites} to emphasize their location in space. The evolution of each unit is prescribed by the same dynamical system, the uncoupled or local dynamic, plus a perturbation given by the interaction with neighbouring sites -- the role of the spatial structure is to define the neighbours of each site.  Coupled map lattices, where the sites are placed at the nodes of a regular  lattice, were studied extensively, see for example \cite{bunimovich1988spacetime,gielis2000coupled,keller2005spectral,keller2006uniqueness} and the references therein. Our current work will focus on the case where the maps are globally coupled, meaning that the neighbours of one unit are all the other units and each site interacts in the same way with every other site so that the system has full permutation symmetry. This model is amenable to taking a particular infinite-sites limit called the \emph{themodynamic limit}, well known from classical mechanics in the continuous time setting \cite{liboff1969introduction}. In our situation, the limit results into a system whose state is given by a probability measure and whose time evolution is given by a \emph{self-consistent operator} \footnote{This is the discrete time analogue of continuous time mean-field models that give rise to the Vlasov equations \cite{braun1977vlasov,dobrushin1979vlasov,maslov1978self}.} (see below for a definition). 

In this paper we study networks where each node of the network (which we will refer to as a \emph{cluster}) is a system of globally coupled maps in the thermodynamic limit. Our goal is to investigate the mechanisms that can lead to synchronization of the states within clusters, meaning that the state of a cluster converges to a Dirac measure by the effect of the dynamics. Below we will make these ideas more precise.


\subsection{Thermodynamic limit for a system of $N$ coupled clusters}
Consider a collection of $n\in \N$ interacting units divided into $N\in \N$ clusters. Given numbers $0< m_1, \,m_2,\,...,\,m_N<1$ with $\sum_i m_i=1$, we assume that for $i<N$, cluster $i$ is made of $M_i(n):=\lfloor m_in\rfloor \in \N$ units, and cluster $N$ is made of the $M_N(n)$ remaining units. Units in the same cluster evolve according to the same deterministic law, which is perturbed by the same type of interactions from units in any other given cluster. This system is described by $n$ coordinates $\bo \xi=(\xi_{1,1},...,\xi_{1,M_1};\,\xi_{2,1},...,\xi_{2,M_2};...)\in \T^n$ whose evolution is given by 
\begin{equation}\label{Eq:}
\xi_{i,j}(t+1)=f_i\circ \Phi_i(\xi_{i,j}(t);\,\bo \xi(t))
\end{equation}
where $f_i:\T\rightarrow \T$ is the uncoupled evolution for units in cluster $i$ and  
\[
\Phi_i(\xi_{i,j};\,\bo\xi)=\xi_{i,j}+\frac1n\sum_{\ell=1}^N\sum_{k=1}^{M_\ell}\frac 1{m_\ell}h_{i \ell}(\xi_{i,j},\xi_{\ell,k})
\]
gives the pairwise coupling between units where the coupling function $h_{i,\ell}$ can vary between different clusters of nodes. Take the thermodynamic limit for $n\rightarrow\infty$ assuming that at a given time  $t\in \N$, for all $i$, there is a probability measure $\mu_i$ such that 
\[
\lim_{n\rightarrow\infty}\frac{1}{M_i}\sum_{j=1}^{M_i}\delta_{\xi_{i,j}(t)}= \mu_i
\]
where the limit is with respect to the weak topology on the space of probability measures of $\T$. Then, if all the $f_i$ and the $h_{i\ell}$ are continuous, the time evolution can be written as 
\[
\xi_{i j}(t+1)=f_i\circ \Phi_{\bo\mu,i}(\xi_{i j}(t))\quad\mbox{where}\quad \Phi_{\bo\mu,i}(x)=x+\sum_{\ell=1}^N\int h_{ i\ell}(x, y)d\mu_\ell(y)
\] 
and where we denoted $\bo\mu=(\mu_1,...,\mu_N)$. 

The transfer operator of a measurable map $T: \T \to \T$ is defined as $T_*: \mc M_1(\T) \to\mc M_1(\T)$ 
\[
T_*\mu=\mu\circ T^{-1}.
\]
By this notation, $f_{i*}$ and $\Phi_{\bo\mu,i*}$ are the transfer operators of the maps $f_i$ and $\Phi_{\bo\mu,i}$ and one has that
\[
\lim_{n\rightarrow\infty}\frac{1}{M_i}\sum_{j=1}^{M_i}\delta_{\xi_{i,j}(t+1)}= f_{i*}\Phi_{\bo\mu,i*}\mu_i
\]
The evolution of the probability measures describing the states of the clusters is therefore given by application of a transfer operator which depends on the measures themselves.\footnote{In the case of one cluster, this construction already appeared in \cite{selley2021linear}.} So in the thermodynamic limit, we have a particular sequential (also called non-autonomous) dynamical system on the circle $\T$. Defining \[(\mc F \bo \mu)_i=f_{i*}\Phi_{\bo\mu,i*}\mu_i, \quad i=1,\dots,N,\] we obtain a nonlinear operator describing the evolution of the system state, which we will call the self consistent transfer operator.

Existence of fixed points for self-consistent operators, their stability, and their behavior under perturbations have been { first investigated by \cite{blank2011self,keller2000ergodic}} and more recently by \cite{BS16,selley2021linear, G21}.  Most of the results cited above have been stated having in mind self-consistent operators arising from small nonlinear perturbations of transfer operators having a spectral gap on some space of absolutely continuous probability measures with densities in a Banach space of regular functions  (e.g. transfer operators for uniformly expanding maps). 

In contrast, the measures we are most interested in are singular with respect to Lebesgue, and the mechanisms that lead to synchronization need/allow for large nonlinear perturbations.  First steps in this direction are contained in \cite{BS16,BKST18} where convergence to a Dirac measure is shown when the strength of interaction is sufficiently strong.

We note that in the continuous time case, synchronization for self-consistent systems has been recently investigated in \cite{bick2020multi}.	

\subsection{Synchronization}
Two units $\xi_{i,j_1}$ and $\xi_{i,j_2}$ in a given cluster are  synchronized if they are in the same state and follow the same evolution. Roughly speaking,  one says that two units synchronize, if the distance between their state variables converges asymptotically to zero. Synchronization is an important concept in applications  being often associated to function or malfunction of  real-world systems (e.g. \cite{hammond2007pathological},  \cite{rohden2012self}). 

Synchronization for  systems of finitely many coupled maps has been extensively investigated. See for example \cite{arenas2008synchronization,pikovsky2003synchronization,koiller2010coupled}. The main objects of study are synchronization manifolds, which are the subsets in phase space where a subset of the units have their state variables equal to each other: in a system with $n$ coupled units having coordinates $\{\xi_1,...,\xi_n\}$, fixed $1\le i_1<i_2<...<i_k\le n$, one can define the associated synchronization manifold as
\[
M_{(i_1,...,i_k)}:=\{\xi_{i_1}=\xi_{i_2}=....=\xi_{i_k}\}.
\]
When the set $M_{(i_1,...,i_k)}$ is invariant under the coupled dynamics, one usually studies conditions for its stability of this manifold. 

In the thermodynamic limit of the cluster model we study \emph{synchronization within each cluster}. We say that a cluster is synchronized if the probability measure describing its state is a Dirac measure. If at the thermodynamic limit a cluster is in the state $\delta_x$, this can be understood as almost every unit  in that cluster (in a suitable probabilistic sense) is in the state $x$. We say that the system is in a \emph{completely synchronized state} if this holds for all clusters. By our definition, not all clusters need to be in the state of the same Dirac measure, $\delta_{x_i}$ is free to depend on cluster $i$, and our general goal is not to discuss synchronization between clusters. However, we will do so in special cases when also synchronization between clusters is expected. We speak of a \emph{partially synchronized state} when only a subset of the clusters are synchronized. Effectively, by synchronization we mean the dynamical stability of synchronized states. We say that a set of states $\mc S$ is dynamically stable, when given a small perturbation of one of these states (from a prescribed set of perturbations), the system converges to an element of $\mc S$ under the effect of the dynamics

\subsection{Organization of the paper} In this paper we give a definition of stability of synchronized states, and discuss criteria implying stability under the evolution of the self-consistent operator. In Sect. \ref{Sec:Setup} we present the setup, give a rigorous definition of synchronized states, and define their stability. In Sect. \ref{Sec:ComplSyncedStates} we give sufficient conditions for stability of completely synchronized states and we describe two mechanisms that can produce stable completely synchronized states: in one the synchronized state is a consequence of the uncoupled dynamics, and in the other it is due to the diffusive nature of the coupling. Most importantly, checking the sufficient conditions that we give only requires the analysis of a finite dimensional dynamical system constructed from the infinite-dimensional self-consistent operator. In Sect. \ref{Sec:PartSyncedStates} we study the stability of partially synchronized states, where the synchronization is due to the influence of the unsynchronized clusters on the clusters that synchronize. We discuss applications on chimera states  and illustrate our result with some numerical simulations.
\\

\subsection*{Acknowledgments} Matteo Tanzi was supported by the Marie Sk{\l}odowska-Curie actions project: ``Ergodic Theory of Complex Systems" proj. no. 843880.  The research of F. M. S\'elley was supported by the European Research Council (ERC) under the European Union’s Horizon 2020 research and innovation programme (grant agreement No 787304).

\section{Setup}\label{Sec:Setup}
\subsection{The self-consistent transfer operator}
We now describe our model considering a network made of $N\in\N$ clusters in the thermodynamic limit. Denote by $\mc M_1$ the set of  Borel probability measures on $\T$. A state of the network is described by an element $\bo \mu=(\mu_1,...,\mu_N)$ in $\mc M_1^N$, where $\mu_i$ is the state of the $i-$th cluster. 

The dynamics is prescribed by maps $\{f_i\}_{i=1}^N$, $f_i:\T\rightarrow \T$ and interaction functions $\{h_{ij}\}_{i,j=1}^N$, $h_{ij}:\T\times \T\rightarrow \R$. { We are going to assume throughout the paper that $\{f_i\}_{i=1}^N$ and $\{h_{ij}\}_{i,j=1}^N$ are smooth maps, in particular twice continuously differentiable in all variables.}

Given $\bo \mu\in\mc M_1^N$, for $i=1,...,N$ define the maps  $\{F_{\bo \mu,i}\}_{i=1}^N$ on $\T$
\begin{equation}\label{Eq:SelfConsistentEq}
F_{\bo \mu,i}(x):= f_i\circ \Phi_{\bo \mu,i}(x)  
\end{equation}
where
\begin{equation}
\Phi_{\bo \mu,i}(x) := x+\sum_{j=1}^N\int_\T h_{ij}(x,y)d\mu_j(y)
\end{equation}
is a mean-field coupling map. The evolution of a state $\bo \mu\in\mc M_1^N$ is given by  $\mc F:\mc M_1^N\rightarrow \mc M_1^N$ defined as
\[
\mc F\bo \mu=(F_{\bo\mu,1*}\mu_1,...,F_{\bo\mu,N*}\mu_N)
\]
where $F_{\bo \mu,i*}$ denotes the transfer operator of the map $F_{\bo \mu,i}$. Later on we are going to denote by $F_{\bo \mu}:\T^N\rightarrow\T^N$ the product map $F_{\bo\mu}(\bo x)=(F_{\bo\mu,1}(x_1),...,F_{\bo\mu,N}(x_N))$ where $\bo x=(x_1,...,x_N)$.

In some cases we are going to write  $\alpha_{ij}h_{ij}$ instead of $h_{ij}$, to highlight the dependence on a parameter $\alpha_{ij} \in \R$ that allows us to ``tune" the \emph{coupling strength} of the interactions from  nodes in cluster $j$ to nodes in cluster $i$.

\subsection{Synchronized states}
As we have already mentioned in the introduction, a cluster is synchronized if its state is a $\delta$-measure concentrated at some point. 


Our main focus will be on particular synchronized states where the support of the Dirac masses is contained in some specific subset of $\T^N$.
\begin{definition}
Let $X \subset \T^N$ and 
\[
\mc S_X:=\{\bo \mu\in\mc M_1^N:\,\,\,\exists \bo x\in X\mbox{ s.t. } \mu_{i}=\delta_{x_i}, \forall i=1,\dots,N\}
\] 
denote the set of \emph{completely synchronized states in $X$}. 
\end{definition}

\begin{definition}
Let $k< N$, $X \subset \T^k$, and 
\[
\mc{PS}^{(i_1,\dots,i_k)}_{X}:=\{\bo \mu\in\mc M_1^N:\,\,\,\exists \bo x\in X\mbox{ s.t. } \mu_{i_j}=\delta_{x_{j}} \forall j=1,\dots,k\}
\]
denote the set of partially synchronized states in $X$ where the clusters $i_1,\dots,i_k$ are synchronized.
\end{definition}
From now on we will assume without loss of generality that in a partially synchronized state, the last $k$ clusters are synchronized, while the first $N-k$ might not be, and we will denote the set of partially synchronized states by
\[
\mc{PS}_{X}=\mc M_1^{N-k}\times \mc S_X.
\]
\begin{remark}
From the definition of the self-consistent transfer operator it follows immediately that  $\mc F \mc S_{\T^N}\subset \mc S_{\T^N}$, and therefore also $\mc F \left(\mc{PS}_{ \T^k}\right)\subset \mc{PS}_{ \T^k}$, implying that synchronized states are invariant under the self-consistent operator. We will see that in many particular cases of interest, $\T^N$ (and $\T^k$) can be restricted to particular subsets $X$, and $\mc M_1^{N-k}$ can be narrowed down to a much smaller set of measures.
\end{remark}

One then wonders if a given synchronized state is stable under $\mc F$, i.e. if perturbing a synchronized state $\bo \mu$ slightly, in a sense that will be made precise below, $\mc F^n\bo\mu$ converges to a synchronized state.   
 
\subsection{Stability of synchronized states}
In the following we denote by $\mc M$ the set of finite signed Borel measures on $\T$, and $d_W$ the Wasserstein distance between measures in $\mc M_1$ (Wasserstein distance associated to the Euclidean metric on $\T$\footnote{ An explicit expression for this distance is given by
\[
d_W(\mu,\nu)=\sup_{\phi\in\Lip_1(\T)}\int \phi d(\mu-\nu)
\]
where $\Lip_1(\T)$ is the set of Lipschitz functions from $\T$ to $\R$ with Lipschitz constant at most 1.}).
 Given $N\in\N$, for  $\bo \mu=(\mu_1,...,\mu_N),\bo \nu=(\nu_1,...,\nu_N)\in\mc M_1^N$ we extend the definition of $d_W$ as follows
\[
d_W(\bo \mu, \bo \nu):=\max_id_W(\mu_i, \nu_i).
\]

For a measure $\mu\in \mc M_1$, we denote by $\supp \mu$ its topological support; for a state $\bo\mu=(\mu_1,...,\mu_N)\in \mc M_1^N$ we denote
\[
\supp \bo\nu:= \supp \mu_1\times....\times \supp\mu_N\subset \T^N,
\]
and we define
\[
|\supp\bo \nu|_{\infty}:=\max_{i}|\supp \mu_i|
\]
with $|\supp \mu_i|:=\inf\{|I|:\, I\mbox{ arc s.t. }\supp\mu_i\subset I\}$ where $|I|$ is the length of the arc $I$.

In the following section  we are going to discuss the stability of synchronized states under perturbation. We mean stability in the following sense: 

{\begin{definition}
	Let  $\bo\mu \in \mc{S}_{X}$ for some $X \subset \T^N$. We say that $\bo\mu$ is a \emph{stable synchronized state} if there is $\epsilon>0$, $R > 0$ such that for every $\bo \mu'\in\mc M_1^{N}$ with $|\supp \bo\mu'|_\infty<\epsilon$ and $\supp \bo \mu' \subset B_R(X)$
	\[
	\lim_{n\rightarrow\infty}d_{W}\left(\mc F^n\bo\mu',\, \mc{S}_{ X}\right)=0.
	\]
\end{definition}
In the partially synchronized case, we mean the following:
\begin{definition}
Let $(\bo\mu,\bo\nu)\in \mc {PS}_{X}$  for some $X \subset \T^{N-N_1}$. We say that $(\bo\mu,\bo\nu)$ is a \emph{stable partially synchronized state} if there are $\mc N \subset \mc M_1^{N_1}$, with $\bo \mu\in \mc N$, and $\epsilon>0$ such that 
\[
\lim_{n\rightarrow\infty}d_{W}\left(\mc F^n(\bo\mu',\bo\nu'),\,  \mc{PS}_X\right)=0.
\]
for every $\bo \mu' \in \mc N$ and $\bo \nu'\in\mc M_1^{N-N_1}$ with $|\supp \bo\nu'|_\infty<\epsilon$ and and $\supp \bo \nu' \subset B_R(X)$.
\end{definition}
\begin{remark}
Notice that $\mc N$ is a set of allowed perturbations for $\bo \mu$. The larger $\mc N$,  the more robust  the stability of the partially  synchronized state. \end{remark}

%
%

\section{Stability of completely synchronized states}\label{Sec:ComplSyncedStates}
We first give a criterion for the stability of completely synchronized states, and then we apply it to obtain a perturbative result on stability of synchronized state in the small coupling regime, and stability of synchronized states in the case of diffusive coupling.
\subsection{Criterion for stability of completely synchronized states}
Consider the map $G:\T^N\rightarrow \T^N$ that describes the action of $\mc F$ on the completely synchronized states $\bo \mu \in\mc S^{(1,...,N)}$ and is defined in the following way: for every $\bo x=(x_1,\dots,x_N)\in\mb T^N$
\begin{equation}\label{Eq:DefG}
 G(\bo x):=\bo x'\,\,\mbox{ where }\,\,\mc F(\delta_{x_1},\dots,\delta_{x_N})=(\delta_{x_1'},\dots,\delta_{x_N'}).
\end{equation}

For every $i=1,...,N$, define functions $g_{i}:\T^N\rightarrow \R$ as  
\begin{equation}\label{Eq:Defgi}
g_i(\bo x):=(F_{(\delta_{x_1},...,\delta_{x_N}),i})'(x_i)=\partial_x\left[f_i\left(x+\sum_{j}h_{ij}(x,x_{ j})\right) \right]|_{x=x_i}.
\end{equation}
These functions control the derivative of $F_{\bo\mu,i}$ for states $\bo\mu$ close to completely synchronized states.

The following result gives a criterion for stability of synchronization in terms of the finite dimensional dynamics of $G$ and the finite set of functions $\{g_i\}_{i=1}^N$. 

\begin{theorem}\label{Thm:CondStabSyncedstates} Assume that there are  $\lambda\in(0,1)$ and $n_0\in\N$ such that 
\begin{itemize}\item[1)] $G$ has an invariant set $X\subset \T^N$ which is uniformly attracting, i.e. $\exists U\supset X$ open such that  $G(U)\subset U$  such that
\[
d(G^{n_0}(\bo x),X)\le \lambda d(\bo x,X)\quad\quad\forall \bo x\in U;
\]
\item[2)] $|\prod_{j=0}^{n_0-1}g_i(G^{j}\bo x)|\le \lambda<1$ for every $\bo x\in X$.
 \end{itemize}

 Then there are $R>0$, $\epsilon>0$, and $\Lambda\in(0,1)$ such that if  $\bo \mu=(\mu_1,...,\mu_N)$ satisfies 
 \[
 |\supp\mu_i|_\infty<\epsilon\quad \mbox{ and }\quad \supp \bo \mu \subset B_R(X):=\left\{\bo y:\, \inf_{\bo x\in X}d(\bo y,\bo x)<R\right\},
 \]
 one has 
 \[
d_W(\mc F^n\bo \mu,\,\mc S_X)\le \mc O(\Lambda^n).
 \]
\end{theorem}
\begin{proof}

We first treat the case $n_0=1$ and then show how to modify the proof for $n_0>1$. We are going to use the notation $\bo \mu^{(n)}=\mc F^n\bo \mu$ for some fixed initial $\bo \mu$.
\\

{\bf Step 1} Here we compare $G$ with the product map $F_{\bo\mu}:\T^N\rightarrow \T^N$
\[
F_{\bo \mu}(\bo x):=(F_{\bo\mu,1}(x_1),...,F_{\bo\mu,N}(x_N)),
\]
and $g_i$ with $(F_{\bo \mu,i})'$ when $\bo \mu$ is  close to a synchronized state.

For every $\delta>0$ there is $\epsilon'>0$ such that if $$\supp \bo \mu\subset B^{\infty}_{\epsilon'}(\bo x):=\{\bo x' \in \T^N: \max_{i}|x_i-x'_i| < \epsilon'\},$$ then $\forall \bo y\in B^{\infty}_{\epsilon'}(\bo x)$
\[
|G(\bo y)-F_{\bo \mu}(\bo y)|<\delta.	
\]
This can be inferred by the equations
\begin{align*}
|(G(\bo y))_i&-(F_{\bo \mu}(\bo y))_i|=|(G(\bo y))_i- F_{\bo\mu,i}(y_i)|\\
&=f_i\left(y_i+\sum_j\alpha_{ij} h_{ij}(y_i,y_j)\right)-f_i\left(y_i+\sum_j\alpha_{ij}\int h_{ij}(y_i,y)d\mu_j(y)\right),
\end{align*}
and by the regularity assumptions on $h$ and $f_i$, that imply that the above quantity can be made as small as wanted by picking $\max_i d_W(\mu_i,\delta_{y_i})$ as small as needed. In particular, $\max_i d_W(\mu_i,\delta_{y_i})$ can be made small by requiring that for all $i$, $\supp \mu_i\subset (x_i-\epsilon',x_i+\epsilon')$ and $y_i\in (x_i-\epsilon',x_i+\epsilon')$, with $\epsilon'>0$ small.

By a similar reasoning, we can deduce that for every $\delta'>0$ there is $\eta>0$ such that if $\supp \bo \mu\subset B^{\infty}_\eta(\bo y)$, then
\[
|g_i(\bo y)-F_{\bo \mu,i}'(z)|<\delta'\quad\quad \forall z\in\supp(\mu_i).
\]

{\bf Step 2}
By assumption 2) and continuity of $g_i$, for every $\Lambda'\in(\lambda,1)$ there is $R>0$ such that   \[|g_i(\bo y)|\le \Lambda'<1,\quad\quad\forall \bo y\in B_R(X)\subset U.\] 
 
This and the conclusion of Step 1 imply that for every $\Lambda\in (\Lambda',1)$ there is $\eta'>0$ sufficiently small such that if $\supp \bo \mu^{(k)}\subset B^\infty_{\eta'}(\bo y^{(k)})$ with $\bo y^{(k)}\in B_{R}(X)$, then
\begin{equation}\label{Eq:ContrmapsF}
|(F_{\bo \mu^{(k)},i})'(z_i^{(k)})|\le \Lambda,\quad\quad\forall \bo z^{(k)}\in \supp(\bo \mu_i^{(k)}).
\end{equation}

{\bf Step 3}  Fix $\Lambda\in(\Lambda',1)$ and let  $\eta'>0$ be as in Step 2. Pick $\delta>0$  sufficiently small so that 
\begin{equation}\label{Eq:Choiceofdelta}
\Lambda R+\delta<R
\end{equation}
 and such that for any $H:\T^N\rightarrow \T^N$ with $|G-H|<\delta$
\begin{equation}\label{Eq:RegEvHs}
H(B_R(X))\subset B_{R}(X).
\end{equation}
Then pick $\epsilon<\eta'$ that satisfies  the conclusion of Step 1 for  $\delta>0$ as above.

With this choice of $\epsilon$, if $\supp\bo \mu^{(n)}\subset B_R(X)$ and $|\supp \bo \mu^{(n)}|_\infty<\epsilon$, then by \eqref{Eq:ContrmapsF} we have that \[
|\supp \bo \mu^{(n+1)}|_\infty<\Lambda|\supp \bo \mu^{(n)}|_\infty<{ \Lambda}^{}\epsilon,
\] 
and since $|G-F_{\bo\mu}|<\delta$
 \[\supp\bo \mu^{(n+1)}=F_{\bo \mu^{(n)}}(\supp\bo \mu^{(n)})\subset B_{\lambda R+\delta}(X)\subset B_{R}(X). \]

{\bf Step 4} Choosing $\bo \mu$ with $\supp\bo\mu\subset B_R(X)\cap B^\infty_\epsilon(X)$ where $R$ and $\epsilon>0$ are as in the steps above, by induction one can show that 
\[
|\supp \bo \mu^{(n)}|_\infty<\Lambda^n|\supp \bo \mu|_\infty.
\]
which implies that $\bo \mu^{(n)}$ converges exponentially fast, w.r.t. $d_W$, to $\mc M_1^{sync}$ as $n\rightarrow +\infty$. 
\\

When $n_0>1$, consider $\tilde{\mc F}:=\mc F^{n_0}$, and notice that 
\[
\tilde {\mc F}(\delta_{x_1},...,\delta_{x_N})=(\delta_{G^{n_0}(x_1)},...,\delta_{G^{n_0}(x_N)}).
\]
So we define $\tilde G:= G^{n_0}$. Furthermore, defining
\begin{equation}\label{Def:tildeF}
\tilde F_{\bo\mu^{(0)},i}(x):=F_{\bo\mu^{(n_0-1)},i}\circ...\circ F_{\bo\mu^{(1)},i}\circ F_{\bo\mu^{(0)},i}(x)
\end{equation}
one has that for $\bo\mu=(\mu_1,...,\mu_N)$, $(\tilde{ \mc F}\bo\mu)_i=\tilde F_{\bo\mu,i*}\mu_i$. We can then consider for all $\bo x=(x_1,..,x_N)\in X$
\[
\tilde g_i(\bo x):= (\tilde F_{(\delta_{x_1},...,\delta_{x_N}),i})'(x_i)=\prod_{j=0}^{n_0-1}g_i(G^j\bo x).
\]
Now it's easy to see that: $\tilde G(X)=X$, $\tilde G(U)\subset U$,
\[
d(\tilde G(\bo x),X)\le \lambda d(\bo x,X)\quad\quad\forall \bo x\in U, 
\]
and  $|\tilde g_i(\bo x)|\le \lambda<1$ for every $\bo x\in X$. One can check that \emph{mutatis mutandis}, steps from 1 to 4  above hold for $\tilde {\mc F}$ in place of $\mc F$ \footnote{The only difference is the  form of $\tilde F_{\bo \mu,i}$, \eqref{Def:tildeF},  compared to that of $F_{\bo\mu,i}$ \eqref{Eq:SelfConsistentEq}.} proving that  there are $R>0$, $\epsilon>0$, and $\Lambda\in(0,1)$ such that if  $\bo \mu=(\mu_1,...,\mu_N)$ satisfies 
 \[
 |\supp\mu_i|_\infty<\epsilon\quad \mbox{ and }\quad \supp \bo\mu \subset B_R(X),
 \] 
 one has 
 \[
d_W(\tilde{\mc F}^n\bo \mu,\,\mc S_X)\le \mc O(\Lambda^n)
 \]
which in turn implies the conclusion of the theorem.
\end{proof}

\begin{remark}
In general, it is not clear whether the orbit of $\mc F^n\bo \mu$ in $\mc M_1^N$ will converge to an orbit in $X$. In the particular case where $X$ is an attracting periodic orbit of $G$, this can be  checked to be true. We are going to show an instance of this in the next section.
\end{remark}

In the following subsection we are going to discuss some applications of the theorem above.

\subsection{Weak coupling} In this section we will consider  interaction functions of the form $\alpha_{ij}h_{ij}$ for some $\alpha_{ij} \in \R$. We talk of weak coupling when $\alpha_{ij}$ are close to zero. Whenever the uncoupled maps $f_i$ have attracting periodic orbits, these give rise to synchronized states that are stable, and stay such when switching on a sufficiently weak coupling.

\begin{definition}
We say that $x \in Y$ is an attracting $k$-periodic point of the map $T \in \mc C^1(Y)$ if $T^k(x)=x$ and there exists $\lambda \in (0,1)$ such that $\|DT^{ k}(x)\| \leq \lambda $. 
\end{definition}

Notice that this implies that $\|DT^{ k}(T^{\ell}(x))\| \leq \lambda $ for all $\ell \in \{0,\dots,k-1\}$ as $DT^{ k}(x)=\prod_{i=0}^{k-1}D(T^{i}(x))=DT^{ k}(T^{\ell}(x))$ for all $\ell \in \{0,\dots,k-1\}$.

In the next theorem we are going to denote $G=G_{\alpha}$ and $g_i=g_{i,\alpha}$, $i=1,\dots,N$ to  highlight the dependence on the coupling strength $\alpha=(\alpha_{ij})_{i,j}$.
\begin{proposition} \label{Thm:PerOrbitSynch}
Let $\bo x \in \T^N$ be an attracting $k$-periodic point for $\bo f=f_1 \times \dots \times f_N:\T^N\rightarrow \T^N$. There is $\epsilon>0$ such that if $|\alpha_{ij}|<\epsilon$ holds for all $i,j \in \{1,\dots,N\}$ we have the following:
\begin{enumerate}
	\item there exists $\bo x_{\alpha} \in \T^N$ such that $G_{\alpha}^k(\bo x_{\alpha})=\bo x_{\alpha}$,
	\item denoting $X_{\alpha}=\{G_{\alpha}^{\ell}(\bo x_{\alpha}):\,\ell=0,...,k-1\}$, there are $R>0$, $\delta>0$ and $\Lambda'\in(0,1)$ such that if  $\bo \mu=(\mu_1,...,\mu_N)$ satisfies 
	\[
	|\supp\mu_i|_\infty<\delta\quad \mbox{ and }\quad \supp \bo\mu\subset B_R\left(X_{\alpha}\right),
	\] 
	one has 
	\[
	d_W(\mc F^n\bo \mu,\, \mc S_{X_{\alpha}})\le \mc O((\Lambda')^n)
	\]
\end{enumerate}
\end{proposition}
\begin{proof}
We prove the statement by checking the assumptions of Theorem \ref{Thm:CondStabSyncedstates}. 

Since $\bo x$ is an attracting periodic point of the smooth map $\bo f=G_0$, there exists a $\rho > 0$ and $\lambda_0 < 1$ such that $\|DG_0^{k}(\bo y)\| \leq \lambda_0 $ for all $\bo y \in B_{\rho}(X_0)$ ( recall that $X_{0}=\{\bo f^{\ell}(\bo x):\,\ell=0,...,k-1\}$).

The map $G_{\alpha}$ writes
\[
(G_{\alpha}(\bo z))_i=f_i\left(z_i+\sum_{j=1}^N \alpha_{ij}h_{ij}(z_i,z_j) \right), \quad i=1,\dots,N
\]
implying that $\alpha \mapsto G_{\alpha} \in \mc C^2$. By the implicit function theorem we get that for all sufficiently small $\varepsilon_1 > 0$, choosing $|\alpha_{ij}| < \varepsilon_1$, $\forall i,j$ there exists $\bo x_{\alpha}$ such that $G_{\alpha}^k(\bo x_{\alpha})=\bo x_{\alpha}$. Furthermore, the distance of $\bo x$ and $\bo x_{\alpha}$ can be made arbitrarily small by decreasing $\alpha_{ij}$. Let us choose $\varepsilon_1$ such that $\|\bo x - \bo x_{\alpha} \|_{\infty} < \rho/2$ for all $|\alpha_{ij}| < \varepsilon_1$, $\forall i,j$. 
 
We are going to show the following:
\begin{enumerate}
	\item $DG_{\alpha}^k(\mathbf{\bo y})$ is uniformly contracting for all $\bo y$ sufficiently close to $X_{\alpha}$. 
	\item $g^k_{i,\alpha}(\bo z)$ is uniformly contracting for all $\bo z \in X_{\alpha}$. 
\end{enumerate}

By the regularity assumptions on $f$ the mapping $\alpha \mapsto DG^k_{\alpha}$ is continuous. This implies that there exists $\varepsilon_2 > 0$ such that for $|\alpha_{ij}| < \varepsilon_2$, $\forall i,j$ there exists a $\lambda_1 < 1$ such that $\|DG_{\alpha}^{k}(\bo y)\| \leq \lambda_1 $ for all $\bo y \in B_{\rho}(X_0)$. This implies that $\|DG_{\alpha}^{k}(\bo y)\| \leq \lambda_1 $ for all $\bo y \in B_{\rho/2}(X_{\alpha})$ and (1) is proved.

As for (2), $\alpha\mapsto g^k_{i,\alpha}$ is continuous, so there exists a $\varepsilon_3 > 0$ such that for $|\alpha_{ij}| < \varepsilon_3$, $\forall i,j$ we obtain $|g_{i,\alpha}^k(\bo z)| \leq \lambda_1$. 

Thus Theorem \ref{Thm:CondStabSyncedstates} proves our statement with the choice of $\lambda=\lambda_1$, $X=X_{\alpha}$ and $U=B_{\rho/2}(X_{\alpha})$.
\end{proof}

\subsection{Diffusive coupling}
In contrast with the previous section where the system exhibited synchronization due to the uncoupled dynamics, in this section we use Theorem \ref{Thm:CondStabSyncedstates} to prove persistence of stable synchronized states  due to strong coupling. In particular, we consider a kind of coupling well known in applications  called \emph{diffusive coupling}, that takes its name from the fact that the resulting dynamics can be viewed as a discrete-time analogue of a generalized reaction-diffusion process. For an overview on diffusive coupling see \cite{chazottes2005dynamics} and the references therein. We talk of diffusive coupling when the interaction function has the form
\[
h(x,y)=\phi(y-x)
\]
where $\varphi:\T\rightarrow \T$ is usually taken to be a function with 
\begin{equation}\label{Eq:Diffusivecoupling}
\phi(0)=0\quad\mbox{and}\quad \phi'(0)>0.
\end{equation} 
This will be a standing assumption throughout this section.

The following results give synchronization criterion in the presence of multiple clusters having equal uncoupled dynamics ($f_i=:f$ for all $i$) with diffusive coupling. In this situation equation \eqref{Eq:SelfConsistentEq} reads
\begin{equation}\label{Eq:LocalMapDiffusive}
F_{\bo \mu,i}(x)=f\left(x+\sum_j\int \phi_{ij}(y-x)d\mu_j(y)\right)
\end{equation}
and the $i-$th coordinate of $G$ -- defined in \eqref{Eq:DefG} -- has expression
\[
(G(\bo x))_i=f\left(x+\sum_j \phi_{ij}(x_j-x_i)\right)
\]
where the functions $\phi_{ij}:\T\rightarrow \T$ satisfy \eqref{Eq:Diffusivecoupling}. 

As can be easily verified, the diagonal set 
\begin{equation}\label{Eq:Diagonal}
\Delta:=\{(x,...,x):\,x\in\T\}\subset \T^{N}
\end{equation}
is invariant under the map $G$. Below,  we can establish conditions for the the existence of  stable synchronized states for the self-consistent operator when $\Delta$ is an attracting set for $G$. Conditions under which $\Delta$ is attracting for $G$ have been extensively studied in the literature of finite dimensional coupled maps\footnote{Notice that $G$ can be seen as the map describing the couple dynamics on a network of  $N$ couple units} (see e.g. \cite{QQQ}).
 
\begin{proposition}\label{Prop:DiffusiveCoupling}
Assume that the set $\Delta:=\{(x,...,x):\,x\in\T\}\subset \T^N$ is an attracting invariant set for $G$ -- as per assumption 1) of Theorem \ref{Thm:CondStabSyncedstates} -- and that
\begin{equation}\label{Eq:DiffusiveCondition}
\max_x |f'(x)| \left|1-\sum_j \phi_{ij}'(0)\right| <1\quad\quad \forall i.
\end{equation}
Then the set of synchronized states $\mc S_{\Delta}$ is stable.
\end{proposition}
\begin{proof}
We should check  condition 2 of Theorem \ref{Thm:CondStabSyncedstates} when $X=\Delta$. From the expression of $g_i$ in \eqref{Eq:Defgi},  plugging in the expression for $F_{\bo \mu,i}$ in \eqref{Eq:LocalMapDiffusive}, we get for $\bo x=(x_1,...,x_N)$
\[
g_i(\bo x):=\left(1-\sum_j\phi'_{ij}(x_j-x_i)\right)\cdot  f'\left(x_i+\sum_j\phi_{ij}(x_j-x_i)\right) 
\] 
and for $\bo x\in \Delta$
\[
g_i(\bo x):= f'\left(x_i\right) \left(1-\sum_j\phi'_{ij}(0)\right).
\]
The condition in \eqref{Eq:DiffusiveCondition} implies that there is $\lambda\in(0,1)$ such that $|g_i(\bo x)|\le \lambda$ for every $\bo x\in \Delta$ and every $i$.
\end{proof}

\begin{remark}
Notice that in the case with one cluster, the dynamics of the self-consistent operator is completely determined by
\[
F_\mu=f\left(x+\int \phi(y-x)d\mu(y)\right).
\]
In this case condition \eqref{Eq:DiffusiveCondition} reads
$
(\max_x|f'(x)|)(1-\phi'(0))<1.
$
\end{remark}

\begin{example}
Assume that $f_i=f \in C^1$ for all $i=1,\dots,N$ and the restriction of $\phi_{ij}$ to an arc $(-\eta,\eta)\subset \T$ satisfies $\phi_{ij}|_{(-\eta,\eta)}(u)=\frac{\alpha}{N} u$, where $\alpha\in \R$ is a uniform coupling strength. We are going to show that one can find some conditions on $\alpha$ such that the assumptions of Proposition \ref{Prop:DiffusiveCoupling} are satisfied.

Consider the set
\[
\Delta_{\omega}=\{\bo x \in \T^N: |x_i-x_j| < \omega, \text{ for all } i,j=1,\dots,N\}.
\]
We now show that for $\alpha$ in a certain range, $G(\Delta_{\eta}) \subset \Delta_{\lambda\eta}$ for some $\lambda < 1$. Indeed, assume that $\bo x\in\Delta_\eta$. Then   $\phi_{m\ell}(x_m-x_{\ell})=\frac{\alpha}{N}(x_m-x_{\ell})$ and $\phi_{m\ell}'(x_m-x_{\ell})=\frac{\alpha}{N}$ for all $m,\ell$. So, calling $x'_j:=(G(\bo x))_j$
\begin{align*}
|x'_m-x_\ell'|&=\left|f\left(x_{\ell}+\sum_{j=1}^N \phi_{\ell j}(x_j-x_\ell) \right)-f\left(x_{m}+\sum_{j=1}^N \phi_{mj}(x_j-x_{m}) \right) \right| \\
&\leq \max|f'|\left|x_\ell-x_{m}+\frac{\alpha}{N}\sum_{j=1}^N(x_j-x_\ell-x_j+x_{m})  \right| \\
&=\max|f'||1-\alpha||x_\ell-x_{m}|
\end{align*}
which is less than $\eta$ if $\lambda:=\max|f'||1-\alpha|<1$. This can always be achieved by choosing $\alpha$ sufficiently close to 1:  notice that in this case is the  \emph{strong coupling} that induces synchronization, not the uncoupled dynamics $f$.

It is also easy to see that \eqref{Eq:DiffusiveCondition}  holds: since $\phi_{ij}'(0)=\frac{\alpha}{N}$
\begin{align*}
(\max_x |f'(x)|) \left|1-\sum_j \phi_{ij}'(0)\right|&=(\max_x |f'(x)|) \left|1-\sum_j\frac\alpha N\right| \\
&=(\max_x |f'(x)|) \left| 1-\alpha\right|\\
&=\lambda.
\end{align*}
\end{example}

\section{Partially Synchronized States}\label{Sec:PartSyncedStates}
We present and motivate the content of this section starting from a simple example illustrating a mechanism by which an unsynchronized cluster can drive another cluster to synchrony, giving rise to a stable partially synchronized state.

\begin{example}\label{Ex:Unsynceddrivsync}
Consider a network with two clusters, 1 and 2, whose self-consistent dynamics at state $(\mu,\nu)\in\mc M_1^2$ is prescribed by the set of equations
\begin{equation}\begin{array}{ll}
F_{(\mu,\nu),1}(x)&= f(x)\\
F_{(\mu,\nu),2}(x)&= f\left(x+ \alpha \phi(x) \int \psi(y) d\mu(y)\right)
\end{array}
\end{equation}
with $f:\T\rightarrow\T$ a $C^2$ uniformly expanding map, i.e. $|f'|\ge\sigma>1$, and where we picked $h_{11}=h_{12}=h_{22}=0$ and \[h_{21}(x,y)=\alpha\phi(x)\psi(y)\] for some $\phi,\psi:\T\rightarrow \R$. 
Notice that the dynamics of cluster 1 does not depend on the state of the system $(\mu,\nu)$, and in particular does not depend on $\nu$, while the evolution of cluster 2 depends on the state of cluster 1. Recall that $f$   has a  unique absolutely continuous invariant probability measure $\bar \mu$ \cite{krzyzewski2004invariant}, and every sufficiently regular density evolves to $\bar \mu$ under iterations of the transfer operator of $f$ (see e.g. \cite{boyarsky2012laws}).

Notice also that  if $\mu$ is such that $\int \psi(y)d\mu(y)=0$, we have that 
\[
F_{(\mu,\nu),2}(x)=f(x),
\] so cluster 2 feels the influence of cluster 1 only in the case where the above integral is nonzero. 

Now assume that $f(0)=0$ and $\phi(0)=0$. Then, $F_{(\mu,\nu),2}(0)=0$ and, in particular, $(\bar \mu, \delta_0)$ is a fixed state for the self-consistent transfer operator. Below we study the stability of this fixed state. 

First of all we find conditions for which $0$ is an attracting fixed point of $F_{(\mu,\nu),2}$ when $\mu=\bar \mu$. We are going to denote  $\bar \mu(\psi):=\int \psi d\bar\mu$ to shorten the notation.  Then
\[
(F_{(\mu,\nu),2})'(x) = f'\left(x+\alpha \phi(x)\bar \mu(\psi)\right)\,(1+\alpha\bar\mu(\psi) \phi'(x))
\]
and evaluating it at zero
\[
(F_{(\mu,\nu),2})'(0) = f'\left(0\right)(1+\alpha\bar\mu(\psi) \phi'(0)).
\]
 Imposing that 0 is attracting we obtain
\[
-1-f'(0)<\alpha\bar\mu(\psi) \phi'(0)f'\left(0\right)<1-f'(0)<0;
\]
if $\bar\mu(\psi) \phi'(0)f'\left(0\right)>0$ then
\begin{equation}\label{Eq:Intervalalfa1}
\alpha\in \left(\frac{-1-f'(0)}{\bar\mu(\psi) \phi'(0)f'\left(0\right)}, \frac{1-f'(0)}{\bar\mu(\psi) \phi'(0)f'\left(0\right)}\right)
\end{equation}
if $\bar\mu(\psi) \phi'(0)f'\left(0\right)<0$ then
\begin{equation}\label{Eq:Intervalalfa2}
\alpha\in \left(\frac{1-f'(0)}{\bar\mu(\psi) \phi'(0)f'\left(0\right)}, \frac{-1-f'(0)}{\bar\mu(\psi) \phi'(0)f'\left(0\right)}\right).
\end{equation}
 \end{example}
It is rather straightforward to show that if $\alpha$ satisfies one of the conditions above, then any state obtained from a sufficiently small perturbation\footnote{Perturbations of $\bar\mu$ should be small and in the Banach space of measure where the transfer operator of $f$ has a spectral gap, perturbations of $\delta_0$ could be any probability measure supported on a sufficiently small neighborhood of $0$. } of $(\bar\mu,\delta_0)$ converges to $(\bar\mu,\delta_0)$ under evolution of the self-consistent transfer operator, and therefore $(\bar\mu,\delta_0)$ is a stable attracting state (this will be made more precise in Section \ref{Sec:ExContinued}). We can interpret this as cluster 2 evolving to a synchronized state (in a stable way) as an effect of the interaction it receives from cluster 1 when the state of  cluster 1 is in the vicinity of $\bar \mu$. 
 
In the remark below we list some conclusions about this example that will be proved in the following sections.
\begin{remark}\label{Rem:ExampleINteract}
 In the example above:
 \begin{itemize}
 \item[i)] Assuming that $\bar\mu(\psi) \phi'(0)\neq 0$, there is an interval for $\alpha$ -- given explicitly in \eqref{Eq:Intervalalfa1} and \eqref{Eq:Intervalalfa2} -- for which the fixed state $(\bar\mu,\delta_0)$ is stable: for any $(\mu,\nu)$ with $\mu$ close to $\bar\mu$ having regular density (e.g. Lipschitz) and $\nu$ being supported on $(-\eta,\eta)$ for a sufficiently small $\eta>0$, $\mc F^n(\mu,\nu)\rightarrow (\bar\mu,\delta_0)$ weakly;
 \item[ii)] if $\alpha=0$ -- i.e. no interaction -- or $\bar\mu(\psi)=0$ -- i.e. no net effect of the interaction --  $F_{(\mu,\nu),2}=f$ and the partially synchronized state $(\bar \mu,\,\delta_0)$ has no hope of being stable; 
 \item[iii)]  adding small interaction terms to cluster 1, i.e.  
 \begin{equation*}\begin{array}{ll}
F_{(\mu,\nu),1}(x)&= f\left(x+ \epsilon\int h_{11}(x,y)d\mu(y)+  \epsilon \int h_{12}(x,y)d\nu(y)\right)\\
F_{(\mu,\nu),2}(x)&= f\left(x+\alpha\int  \phi(x)\psi(y)d\mu(y)\right)
\end{array}
\end{equation*}
(where $h_{11}(0,y)=h_{12}(0,y)=0$) preserves the stability of a partially synchronized state. The main difficulty here is that perturbing the synchronized state $\delta_0$ can potentially change the evolution of cluster 1 in such a way that its state gets away from $\bar\mu$ under evolution, and in turn this can destroy the stability of $\delta_0$. We are going to show that this cannot happen whenever $\epsilon$ is sufficiently small.
 \end{itemize}
 \end{remark}
Below we provide a framework to make all of this precise and prove points i) and iii) of the remark above (see Sect. \ref{Sec:ExContinued}).

\subsection{A criterion for stability of partially synchronized states}
In this section, unless specified otherwise, we consider a network of $N$ clusters where the clusters are divided into two groups, Group 1 and Group 2,  one made of $N_1$  and the other of $N_2=N-N_1$ clusters. Without loss of generality we can assume that the clusters in Group 1 have been labeled $\{1,...,N_1\}$ and those in Group 2 have labels $\{N_1+1,...,N\}$. 

Define $\Pi_i:\mc M_1^N\rightarrow \mc M_1^{N_i}$ the projections on the first $N_1$ coordinates and last $N_2$ coordinates respectively. After fixing $\bo \nu\in\mc M_1^{N_2}$, one can define $\mc F_{1,\bo\nu}:\mc M_1^{N_1}\rightarrow \mc M_1^{N_1}$ as $\mc F_{1,\bo\nu}\mu:= \Pi_1\mc F(\bo\mu,\bo\nu)$. $\mc F_{2,\bo\mu}:\mc M_1^{N_2}\rightarrow \mc M_1^{N_2}$ is defined analogously.

To fix ideas, we think of the second group of clusters as those that synchronize, while those in the first group might be unsynchronized. Under this perspective, for every  fixed $\bo \mu\in\mc M_1^{N_1}$, define $G_{\bo \mu}:\T^{N_2}\rightarrow \T^{N_2}$ and $g_{\bo \mu,i}:\T^{N_2}\rightarrow \R$ as in \eqref{Eq:DefG} and \eqref{Eq:Defgi} for the self-consistent operator $\mc F_{2,\bo \mu}$. 

\begin{theorem}\label{Thm:PartialSynchronization}
Consider a network with two groups of clusters as above, and suppose that there are $\mc N'\subset\mc M_1^{N_1}$, $U\subset \T^{N_2}$,  $\epsilon_0>0$, $\delta>0$, and $\lambda\in[0,1)$ such that 
\begin{itemize}
\item[A1)]   for every $\bo \mu\in \mc N'$, $G_{\bo\mu}(U) \subset U$,  $d(G_{\bo\mu}(\bar U),\partial U)>\delta$ and $|g'_{\bo \mu,i}(x)|\le \lambda$ for all $x\in U$;
\item[A2)] letting $\mc S_{\epsilon,U}:= \{ \bo\nu\in \mc M_1^{N_2}: \,|\supp \bo\nu|_\infty<\epsilon, \,\, \supp\bo\nu\subset U\}$, there is $\mc N\subset \mc N'$ s.t.
\[
\Pi_1\mc F^n (\mc N\times \mc S_{\epsilon,U} )\subset \mc N',\quad\quad\forall n\in\N, \quad \forall \epsilon \in (0,\epsilon_0].
\]
\end{itemize} 
Then there is $\epsilon\in(0,\epsilon_0]$ such that  if $(\bo\mu,\bo\nu)\in\mc N\times \mc S_{\epsilon,U}$, then
\begin{equation}\label{Eq:THm4Thes}
\lim_{n\rightarrow \infty}d_W\left( \mc F^n(\bo\mu,\bo\nu), \, \mc {PS}_U\right)=0,
\end{equation}
\end{theorem}
Assumption A1) ensures that  the clusters in Group 2 synchronize as long as the state of Group 1 is controlled and belongs to a given set $\mc N$. Assumption A2) requires that when the clusters of Group 2 are close to synchrony, the clusters in Group 1 evolve in a controlled way and their state remains inside $\mc N$. A typical situation we have in mind is when, for $\nu$ close to a synchronized state, the maps $F_{(\bo \mu,\bo\nu),i} $ for $i=1,...,N_1$ are close to maps $\bar F_i$ with a spectral gap whose invariant measure is stable under perturbations. It is known that in many such situations, an arbitrary composition of maps which are all close to  given statistically stable maps $\bar F_i$, keep the measures close to the invariant measures of $\bar F_i$ (see e.g. \cite{tanzi2019robustness}.)
\begin{proof}[Proof of Theorem \ref{Thm:PartialSynchronization}]
For  $\epsilon>0$ take  $\bo\nu\in\mc M_1^{N_2}$ with $|\supp\bo\nu|_{\infty}<\epsilon$ and $\supp\bo\nu\subset U$. It immediately follows from A1) that $|\supp\mc F_{\bo\mu,2}\bo\nu|_{\infty}<\lambda\epsilon$.  
Arguing as in Step 1 of the  proof of Theorem \ref{Thm:CondStabSyncedstates}, one can pick  $\epsilon>0$ sufficiently small, so that 
\[
d(G_{\bo\mu}(\bo x), F_{\bo\mu,2}(\bo x))<\frac\delta2,\quad\quad\forall \bo x\in\T^{N_2}
\]
where $F_{\bo\mu,2}=(F_{\bo\mu,N_1+1},...,F_{\bo\mu,N})$.
Therefore, if $|\supp\bo\nu|_{\infty}<\epsilon$ 
\[
d(\supp \mc F_{\bo\mu,2}\bo\nu, \partial U)= d(F_{\bo\mu,2}(\supp\nu),\partial U) \ge d(G_{\bo\mu}(\supp\nu),\partial U)-\frac\delta2=\frac\delta2,
\]
and  $\supp \mc F_{\bo\mu,2}\bo\nu\subset U$.

It follows from A2) that if $(\bo\mu,\bo\nu)$ are such that $\bo\mu\in\mc N$ and $|\supp\bo\nu|_{\infty}<\epsilon$ and $\supp\bo\nu\subset U$, then by A1) $\Pi_1\mc F_{(\bo\mu,\bo\nu)}^n(\bo\mu,\bo\nu)\in\mc N'$ and $|\supp\Pi_2\mc F^n_{(\bo\nu,\bo\mu)}(\bo\mu,\bo\nu)|_{\infty}<\lambda^n\epsilon$, $\supp\Pi_2\mc F^n_{(\bo\nu,\bo\mu)}(\bo\mu,\bo\nu)\subset U$ for every $n\in \N$ which implies \eqref{Eq:THm4Thes}.
\end{proof}

This theorem has mostly a descriptive purpose as the conclusion immediately follows from the assumptions. In the next section we show how to verify assumptions A1) and A2)  to prove points i) and iii) of Remark \ref{Rem:ExampleINteract}, and in Sect. \ref{Sec:Chimeras} we present an application to chimera states. 

\

\subsection{Example \ref{Ex:Unsynceddrivsync} continued}\label{Sec:ExContinued}
In this section we apply Theorem \ref{Thm:PartialSynchronization} to prove points i) and iii) of Remark \ref{Rem:ExampleINteract}.

\begin{proof}[Proof of Remark \ref{Rem:ExampleINteract} point i)]
We are going to exploit the spectral properties of $f$. Consider the Banach space $\mc C^1$ of absolutely continuous finite signed measures with densities in  $C^1(\T,\R)$ . We endow this space with the following norm: for a measure $d\mu=\rho dm$ we define $\|\mu\|=\int|\rho'|+\int|\rho|$. We further consider $L^1$ as a weak norm $|\cdot|_{L^1} \leq \|\cdot\|$ on this space as $|\mu|_{L^1}=\int|\rho|$. It is known that the transfer operator $f_*$ has a spectral gap on this Banach space, i.e. there is $\bar\mu\in \mc C^1$ such that $f_*\bar\mu=\bar\mu$, and there are $C>0$ and $\lambda\in (0,1)$ satisfying 
\begin{equation}\label{Eq:SpecGap}
\|f_*^n\mu\|\le C\lambda^n\|\mu\|,\quad\quad\forall\mu\in  \mc C^1_0
\end{equation}
where $\mc C^1_0=\{\mu\in \mc C^1:\,\mu(\T)=0\}$. 
Fix any $\lambda'\in (\lambda,1)$ and define  the norm $\|\cdot\|'$ on $ \mc C^1_0$ by
\[
\|\mu\|'= \sum_{k=0}^\infty \frac{\|f_*^n\mu\|}{(\lambda')^n}
\]
which is well defined in virtue of \eqref{Eq:SpecGap}. Notice that 
\begin{equation}\label{Eq:Equivalenceofnorms}
\|\mu\|\le \|\mu\|'\le  \frac{C}{1-\frac{\lambda}{\lambda'}} \|\mu\|
\end{equation}
and so $\|\cdot\|'$ and $\|\cdot\|$ are equivalent. Let us call $\tilde C:=  \frac{C}{1-\frac{\lambda}{\lambda'}}$. The advantage of working with $\|\cdot\|'$ is that $f_*$ contracts in one step with respect to this norm:
\begin{equation}\label{Eq:Contractionprimednorm}
\|f_*\mu\|'= \sum_{k=0}^\infty \frac{\|f_*^{n+1}\mu\|}{(\lambda')^n}= \lambda'\|\mu\|'.
\end{equation}

Now we proceed to verify the assumptions of Theorem \ref{Thm:PartialSynchronization}. Notice that $F_{(\mu,\nu),2}=F_{\mu,2}$ since $F_{ (\mu,\nu),2}$ does not depend on $\nu$. Recall that $F_{\mu,2}(0)=0$ for every  $\mu$ and that, with the assumptions on $\phi$ and $\psi$, and the choice of $\alpha$, $|(F_{\bar \mu,2})'(0)|<1$. Now, 
\begin{align*}
| F_{\mu,2}'(0)-F_{\bar \mu,2}'(0)| &= \left| f'\left(0\right)(1+\alpha\mu(\psi) \phi'(0))-f'\left(0\right)\left(1+\alpha\bar\mu(\psi) \phi'(0)\right)\right|\\
&=|\alpha f'\left(0\right)\phi'(0)||\mu(\psi)-\bar\mu(\psi)|.
\end{align*}
By  H\"older inequality we get that 
\[
|\mu(\psi)-\bar\mu(\psi)|\le \sup_x|\psi(x)|\, |\mu-\bar\mu|_{L^1}.
\]
Therefore, one can find $\delta_1>0$ such that if $|\mu-\bar\mu|_{L^1}<\delta_1$
\[
| F_{\mu,2}'(0)|<1.
\]
With an analogous reasoning involving the second derivative, we conclude that, possibly decreasing $\delta_1>0$, there is all neighbourhood $I=[-\Delta,\Delta]\subset \T$ of $0$ such that 
\begin{equation}\label{Eq:unifConFmu1}
\left| F_{\mu,2}'|_{I}\right|<1
\end{equation}
for every $\mu$ such that $|\mu-\bar \mu|_{L^1}<\delta_1$. 

Pick $\mc N':=\{\mu\in\mc B\cap \mc M_1:\,\|\mu-\bar\mu\|'<\delta_1 \}$ and $U=I$. It follows from the first inequality in \eqref{Eq:Equivalenceofnorms} together with \eqref{Eq:unifConFmu1} that for every $\mu\in \mc N'$
\[
G_{\mu}(U)=F_{\mu,2}(I)\subset I=U
\]
and that $d(G_{\mu}(U),\partial U)>\delta$ for some $\delta>0$\footnote{If $| F_{\mu,2}'|_{I}|\le \lambda_1<1$, then $\delta=(1-\lambda_1)\Delta$. }. Which proves that this choice of $\mc N'$ and $U$ satisfies assumption A1). By \eqref{Eq:Contractionprimednorm},  $\|f_*\mu-\bar\mu\|'=\|f_*\mu-f_*\bar\mu\|'<\lambda'\delta_1$, therefore
\[
\mc F_{1,\nu}(\mc N')=f_*(\mc N')\subset \{\mu\in\mc B\cap \mc M_1:\, \|\mu-\bar \mu\|'<\lambda'\delta_1\}\subset \mc N'.
\]
Moreover, $|\mu-\bar \mu|_{L^1}\le\|\mu-\bar \mu\|'$ and therefore \eqref{Eq:unifConFmu1} is satisfied for every $\mu\in\mc N$. Picking $\mc N=\mc N'$,  assumption A2) is proved.
\end{proof}

\begin{proof}[Proof of Remark \ref{Rem:ExampleINteract} point iii)] 
Define the  distance $d_{\mc B}$ between linear operators $\mc P_1,\mc P_2:\mc B\rightarrow \mc B$
\[
d_{\mc B}(\mc P_1,\mc P_2):=\sup_{\eta\in \mc B:\|\eta\|\le 1}|\mc P_1\eta-\mc P_2\eta|_{L^1}.
\]


Remember that $F_{(\mu,\nu),1}=f\left(x+ \epsilon\int h_{11}(x,y)d\mu(y)+  \epsilon \int h_{12}(x,y)d\nu(y)\right)$ according to point iii) of Remark \ref{Rem:ExampleINteract}.

The following lemma shows that the transfer operator of $F_{(\mu,\nu),1}$ and $f$ are close, more precisely at a distance of order $\epsilon$.
\begin{lemma}\label{Lem:Lemmafdistance} There is $C_1>0$ such that  
\[
d_{\mc B}((F_{(\mu,\nu),1})_*,f_*)<C_1\epsilon.
\]
\end{lemma}
\begin{proof}	
By \cite[Lemma 13]{K82}, we have $d_{\mc B}((F_{(\mu,\nu),1})_*,f_*) \leq 12d(F_{(\mu,\nu),1},f)$, where
\begin{align*}
d(T_1,T_2)=\inf \{&\kappa > 0 : \exists A \subset \T \text{ and } \sigma: \T \to \T \text{ such that } Leb(A) > 1-\kappa \\
& \sigma \text{ is a diffeomorphism, } T_2|_A=T_1 \circ \sigma|_A \text{ and for all } x\in \T: \\
& |\sigma(x)-x| < \kappa, |1/\sigma'(x)-1| < \kappa\}.
\end{align*}
Now since $F_{(\mu,\nu),1} = f \circ \Phi_{(\mu,\nu),1}$ where 
\[
\Phi_{(\mu,\nu),1}(x)=x+ \epsilon\int h_{11}(x,y)d\mu(y)+  \epsilon \int h_{12}(x,y)d\mu_2(y)
\]
is a diffeomorphism of $\T$ (for $\varepsilon$ sufficiently small) and we have the bounds
\begin{align*}
|\Phi_{(\mu,\nu),1}(x)-x| &\leq \varepsilon(\sup|h_{11}|+\sup|h_{12}|) \\
|1/\Phi'_{(\mu,\nu),1}(x)-1| &\leq \varepsilon\cdot \frac{\sup|\partial_xh_{11}|+\sup|\partial_xh_{12}|}{1-\varepsilon(\sup|\partial_xh_{11}|+\sup|\partial_xh_{12}|)} \\
\end{align*}
we can conclude that $d(F_{(\mu,\nu),1},f) \leq C_1\varepsilon$ and by consequence $d_{\mc B}((F_{(\mu,\nu),1})_*,f_*) \leq C_1\varepsilon$, which proves the lemma.
\end{proof}

We now use the above result to prove that when $\epsilon>0$ is sufficiently small, the maps $F_{(\mu,\nu),1}$ satisfy Lasota-Yorke inequalities with uniform constants. 
\begin{lemma}\label{Lem:UnifLYINeq}
There exist   $\alpha \in (0,1)$, $K > 0$, and $\varepsilon_1 > 0$ such that when $|\varepsilon| < \varepsilon_1$
\begin{align*}
\|(F_{(\mu,\nu), 1})_*\eta\| &\leq  \alpha\|\eta\|+K|\eta|_{L^1}.
\end{align*}
for all $\eta \in \mc C^1$ and $(\mu,\nu) \in \mc M_1^2$.
\end{lemma}

\begin{proof}
Denote the density of $\eta$ by $\rho$, the inverse branches of $F_{(\mu,\nu), 1}$ by $ (F_{(\mu,\nu), 1})_i^{-1}$, and the action of $(F_{(\mu,\nu), 1})_*$ on  $C^1$ densities by  $\mathcal L_{F_{(\mu,\nu), 1}}: C^1(\T,\R)\rightarrow C^1(\T,\R)$. Then
\[
\mathcal L_{F_{(\mu,\nu), 1}}\rho=\sum_i \frac{\rho}{|(F_{(\mu,\nu), 1})'|} \circ (F_{(\mu,\nu), 1})_i^{-1}
\]
and
\begin{align*}
&\int |(\mathcal L_{F_{(\mu,\nu), 1}}\rho)'|=\int\left | \left(\sum_i \frac{\rho}{|(F_{(\mu,\nu), 1})'|} \circ (F_{(\mu,\nu), 1})_i^{-1}\right)' \right| \\
&\leq \int \left |\sum_i \frac{\rho'}{| (F_{(\mu,\nu), 1})'|^2} \circ  (F_{(\mu,\nu), 1})_i^{-1} \right|+\int \left |\sum_i \frac{\rho \cdot ( F_{(\mu,\nu), 1})''}{| (F_{(\mu,\nu), 1})'|^3} \circ  (F_{(\mu,\nu), 1})_i^{-1} \right| \\
&\leq  \int \left | \frac{\rho'}{ (F_{(\mu,\nu), 1})'} \right|+\int \left | \frac{\rho \cdot  (F_{(\mu,\nu), 1})''}{( (F_{(\mu,\nu), 1})')^2} \right|
\end{align*}
We can compute that for $\varepsilon$ sufficiently small we have 
\begin{align*}
|(F_{(\mu,\nu), 1})'| \geq \:& \omega(1-\varepsilon(\sup|\partial_xh_{11}|+\sup|\partial_xh_{12}|)) \geq  \frac{1}{\alpha} > 1 \\
|(F_{(\mu,\nu), 1})''|  \leq \:& \sup|f''|(1+\varepsilon(\sup|\partial_xh_{11}|+\sup|\partial_xh_{12}|))^2 \\
&+\varepsilon \cdot \sup|f'|(\sup|\partial_x^2h_{11}|+\sup|\partial_x^2h_{12}|) \leq D
\end{align*}	
for some $\alpha \in (0,1)$ and $D>0$. This gives
\[
\int |(\mathcal L_{F_{(\mu,\nu), 1}}\rho)'| \leq  \alpha \int |\rho'|+ D \alpha^2\int|\rho|_{}
\]
and finally
\[
\|(F_{(\mu,\nu), 1})_*\eta\| \leq \alpha \|\eta\|+K|\eta|_{L^1}
\]
with $K=1+D\alpha^2$.
\end{proof}

Given $(\mu^{(0)},\nu^{(0)})\in\mc M_1^2$, let $(\mu^{(n)},\nu^{(n)}):=\mc F^n(\mu^{(0)},\nu^{(0)})$ its orbit under the self-consistent transfer operator. 

\begin{lemma}\label{Lem:UniformBoundStrongNorm}
Assume that $|\epsilon|<\epsilon_1$, with $\epsilon_1$ as in Lemma \ref{Lem:UnifLYINeq}. There is $K>0$ such that if $\|\mu^{(0)}\|\le K$, then $\|\mu^{(k)}\|\le K$ for every $k\in\N$. In particular, $\|\bar \mu\|\le K$.
\end{lemma}
\begin{proof}
Immediately follows from the fact that 
\[
\mu^{(k)}=(F_{ (\mu^{(k-1)},\nu^{(k-1)}),1})_*...(F_{ (\mu^{(0)},\nu^{(0)}),1})_*\mu^{(0)}\]
 and that by Lemma \ref{Lem:UnifLYINeq} $(F_{ (\mu^{(i)},\nu^{(i)}),1})_*$ all satisfy Lasota-Yorke inequalities with uniform constants.
\end{proof}
For every $k\in\N\cup\{0\}$
\begin{align}
\left| \mu^{(k)}-\bar \mu \right|_{L^1}&=\left| (F_{ (\mu^{(k-1)},\nu^{(k-1)}),1})_*...(F_{ (\mu^{(0)},\nu^{(0)}),1})_*\mu^{(0)}-f_*^k\bar \mu \right|_{L^1}\nonumber\\
&\le |f_*^k(\mu^{(0)}-\bar \mu)|_{L^1}+\sum_{j=0}^{k-1} \left|(F_{ (\mu^{(k-1)},\nu^{(k-1)}),1})_*...\left[ (F_{ (\mu^{(j)},\nu^{(j)}),1})_*-f_*\right]f^{j-1}_*\bar \mu\right|_{L^1}\label{Eq:EstL1normoneway}\\
&\le C\lambda^k\|\mu^{(0)}-\bar \mu\|+k\delta_0\|\bar \mu\| \label{Eq:EstL1normtwoway}
\end{align}
where going from \eqref{Eq:EstL1normoneway} to \eqref{Eq:EstL1normtwoway} we used
\[
|f_*^k(\mu^{(0)}-\bar \mu)|_{L^1}\le \|f_*^k(\mu^{(0)}-\bar \mu)\| \le C\lambda^k\|\mu^{(0)}-\bar \mu\|
\]
and
\begin{align*}
\left|(F_{ (\mu^{(k-1)},\nu^{(k-1)}),1})_*...\left[ (F_{ (\mu^{(j)},\nu^{(j)}),1})_*-f_*\right]f^{j-1}_*\bar \mu\right|_{L^1} &\le \left|[ (F_{ (\mu^{(j)},\nu^{(j)}),1})_*-f_*]\bar \mu\right|_{L^1}\\
&\le \delta_0\|\bar \mu\|.
\end{align*}
Fix $n\in\N$ such that $2C\lambda^n K<\frac{\delta_1}{4}$, where $K$ is as in Lemma \ref{Lem:UniformBoundStrongNorm} and $\delta_1$ such that \eqref{Eq:unifConFmu1} is satisfied when $|\mu-\bar\mu|_{L^1}<\delta_1$. Then pick $\delta_0$ so that $n\delta_0\|\bar \mu\|<\frac{\delta_1}{4}$. With these choices
\[
\left| \mu^{(n)}-\bar \mu \right|_{L^1}<\frac{\delta_1}{2},
\]
and arguing by induction (redefining $(\mu,\nu)^{(0)}:=(\mu,\nu)^{((m-1)n)}$)
\[
\left| \mu^{(mn)}-\bar \mu \right|_{L^1}<\frac{\delta_1}{2}
\]
for every $m\in \N$.
Now, picking $|\mu^{(0)}-\bar \mu|_{L^1}<\frac{\delta_1}{2}$, for every $m\in \N$ and $q\in\{1...,n-1\}$, using $\eqref{Eq:EstL1normoneway}$,
\begin{align*}
\left| \mu^{(mn+q)}-\bar \mu \right|_{L^1}&\le   |f_*^q(\mu^{(mn)}-\bar \mu)|_{L^1}+q\delta_0\|\bar \mu\|\\
&\le |\mu^{(mn)}-\bar \mu|_{L^1}+n\delta_0\|\bar \mu\|\\
&<\frac{\delta_1}{2}+\frac{\delta_1}{4}\\
&<\delta_1.
\end{align*}

Assumptions A1) and A2) from Theorem \ref{Thm:PartialSynchronization} are  satisfied picking $\mc N:=\{\mu\in \mc C^1\cap \mc M^1:\,\,\|\mu\|\le K,\, |\mu-\bar\mu|_{L^1}<\delta_1/2\}$ and $\mc N':=\{\mu\in \mc C^1\cap \mc M^1:\,\,\|\mu\|\le K,\, |\mu-\bar\mu|_{L^1}<\delta_1\}$.

\end{proof}

\subsection{Chimera States}\label{Sec:Chimeras}

There is no consensus on a mathematical rigorous definition of a chimera state. Loosely speaking one can describe a chimera in the following way. Consider a system of finitely many interacting units.  If the structure of the interaction has  some symmetry, a \emph{chimera state} is a persistent state of the network  that breaks this symmetry. For example, consider a system of $n$ globally coupled units described by the variables $(\xi_1,...,\xi_n)\in \T^n$ with time evolution given by
\begin{equation}\label{Eq:CoupledSystFullSymm}
\xi_i(t+1)= f\left(\xi_i(t)+\frac 1n\sum_{j=1}^n h(\xi_i(t),\xi_j(t))\right).
\end{equation}
In this case, every unit is indistinguishable from all the other units: the system has full permutation symmetry. Then, for example, a state for this system where part of the coordinates are synchronized and  part are unsynchronized, and such that this distinction persists under the time evolution can be called a chimera state. 

 Chimeras have been studied in systems of coupled maps (\cite{abrams2004chimera}, \cite{bick2016chaotic} among many others), and observed in real world systems \cite{martens2013chimera}. 
 
 Here we show  how chimera states arise and can be described in the framework of self-consistent transfer operators. Consider a system as in \eqref{Eq:CoupledSystFullSymm}. Fix $\ell\in [0,1]$, for every $n\in\N$, divide the units into two groups of $\lfloor \ell n\rfloor$ and $n-\lfloor \ell n\rfloor$ units respectively. Assume that the units $\{\xi_i\}_{i=1}^{\lfloor n\ell\rfloor}$ are distributed according to $\mu_1$  while $\{\xi_{i}\}_{i=\lfloor n\ell \rfloor +1}^{k}$ are distributed according to the measure $\mu_2$, meaning that 
 \[
 \lim_{n\rightarrow \infty}\frac1n\sum_{j=1}^{\lfloor n\ell \rfloor} \delta_{\xi_j}=\ell\mu_1,\quad\mbox{and}\quad
 \lim_{n\rightarrow \infty}\frac1n\sum_{j=\lfloor n\ell \rfloor+1}^{n} \delta_{\xi_j}=(1-\ell)\mu_2
 \]
 where convergence is with respect to the weak topology.
  Then in the thermodynamic limit, if $h$ is continuous, for every $\xi$
 \begin{align*}
 \lim_{n\rightarrow \infty}f\left(\xi+\frac 1n\sum_{j=1}^nh(\xi,\xi_j)\right)&=\lim_{n\rightarrow \infty}f\left(\xi+ \left[\frac1n\sum_{j=1}^{\lfloor n\ell \rfloor}h(\xi,\xi_j)+\frac1n\sum_{j=\lfloor n\ell \rfloor+1}^{n}h(\xi,\xi_j)\right]\right)\\
 &= f\left(\xi+\ell \int h(\xi,y)d\mu_1(y)+(1-\ell)\int h(\xi,y)d\mu_2(y)\right).
 \end{align*}
 
The above defines the self-consistent transfer operator on two clusters associated to 
\begin{equation}\label{Eq:ChimeraSplitting}
F_{(\mu_1,\mu_2),i}(x)=f\left(x+\ell \int h(x,y)d\mu_1(y)+(1-\ell)\int h(x,y)d\mu_2(y)\right)
\end{equation}
independent of $i$. $(\mu_1,\mu_2)$ is a chimera state for this system if $\mu_1\neq\mu_2$ and $\mu_1$ and $\mu_2$ are measures fixed by $F_{(\mu_1,\mu_2),i}=F_{(\mu_1,\mu_2)}$. 

{We now give two examples. The understanding of stability in both examples goes beyond the statement of the main theorem on partially synchronized states, as the stabilty of the unsynchronized state is also discussed.}

In the first one, part of the network converges to a fixed state given by a single point, while the other part has a fixed state supported on two points. This phenomenon is also known as \emph{dynamical clustering}). 

\begin{example} \label{ex:chimera}
Let $f(x)=2x \mod 1$ and $h(x,y)=-\frac{1}{10\pi}\sin(6\pi x)\cos(6\pi y)$. Let $\mu_{1}:=\frac{1}{2}(\delta_{1/3}+\delta_{2/3})$ and $\mu_2:=\delta_0$. Notice that $h(0,y)=h(1/3,y)=h(2/3,y)=0$ and this implies that
$F_{(\mu_1,\mu_2)}(0)=f(0)=0$, $F_{(\mu_1,\mu_2)}(1/3)=f(1/3)=2/3$ and
$F_{(\mu_1,\mu_2)}(2/3)=f(2/3)=1/3.$ Then $(F_{(\mu_1,\mu_2)})_*\left(\frac{1}{2}(\delta_{1/3}+\delta_{2/3}) \right)=\frac{1}{2}(\delta_{1/3}+\delta_{2/3})$ and $(F_{(\mu_1,\mu_2)})_*\delta_0=\delta_0$, so $(\mu_1,\mu_2)$ is a chimera state. 

Compute that
\begin{align*}
&F_{(\nu_1,\nu_2)}'(x)\\
&=2\left(1+\ell \int \partial_xh(x,y)d\nu_1(y)+(1-\ell)\int \partial_xh(x,y)d\nu_2(y)\right)\\
&=2\left(1-\frac{6\ell}{10} \int \cos(6\pi x)\cos(6\pi y)d\nu_1(y)-\frac{6(1-\ell)}{10}\int \cos(6\pi x)\cos(6\pi y)d\nu_2(y)\right).
\end{align*}

Assume that $(\nu_1,\nu_2)$ is such that $\supp \nu_1 \subset B_{r}(1/3)\cup B_{r}(2/3)$ and $\supp \nu_1 \subset B_{r}(0)$. Notice that $\cos(6\pi x)\cos(6\pi y)=1$ if $x^*,y^* \in \{0,1/3,2/3\}$. So fixing $r$ sufficiently small we can achieve that $F_{(\nu_1,\nu_2)}|_{B_{r}(0)}<9/10$, $F_{(\nu_1,\nu_2)}|_{B_{r}(1/3)}<9/10$ and $F_{(\nu_1,\nu_2)}|_{B_{r}(2/3)}<9/10$.  Now we easily see that the conditions of Theorem \ref{Thm:PartialSynchronization} hold with $U=B_{r}(0)$ and $\mc N'=\mc N=\{\mu \in \mc M_1: \supp \mu \subset B_{r}(1/3)\cup B_{r}(2/3)\}$. This gives us that $d_W(\mc F^n (\nu_1,\nu_2),\mc {PS}_{U}) \to 0$.

But in fact much more is true. By using the notation $\nu_1^{(1)}:=(F_{(\nu_1,\nu_2)})_*\nu_1$ and $\nu_2^{(1)}:=(F_{(\nu_1,\nu_2)})_*\nu_2$ it will hold that $\supp \nu_1^{(1)} \subset B_{9/10 \cdot r} (1/3)\cup B_{9/10 \cdot r} (2/3)$ and $\supp \nu_2^{(1)} \subset B_{9/10 \cdot r} (0)$. Iterating this gives that $$d_W((\nu_1^{(n)},\nu^{(n)}_2),(\mu_1,\mu_2)) \to 0$$ as $n \to \infty$.
\end{example}

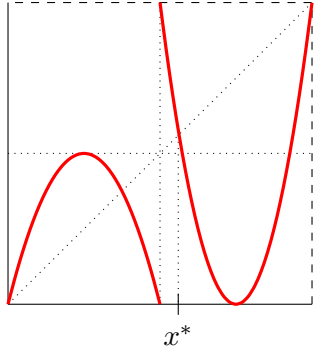
\begin{figure}
	\centering
	\begin{tikzpicture}[scale=4]
	\draw (0,0) -- (1,0);
	\draw (0,0) -- (0,1);
	\draw[dashed] (1,0) -- (1,1) -- (0,1);
	\draw[dotted] (0.5,0) -- (0.5,1);
	\draw[dotted] (0,0.5) -- (1,0.5);
	\draw[dotted] (0,0) -- (1,1);
	\draw[dotted] (0.56,0.56) -- (0.56,0);
	\draw[domain=0:0.5, smooth, variable=\x, red, very thick] plot ({\x}, {4*\x*(1-2*\x)});
	\draw[domain=0.5:1, smooth, variable=\x, red, very thick] plot ({\x}, {-8*(\x-0.5)*(2-2*\x)+1});
	\foreach \x/\xtext in {0.56/x^*}
	\draw[shift={(\x,0)}] (0pt,1pt) -- (0pt,-1pt) node[below] {$\xtext$};
	\end{tikzpicture}
\caption{The graph of $f$ in Example \ref{ex:chimera2}.}
\label{fig:chimera2}
\end{figure}

In the following example we sketch how to obtain a chimera state for a self-consistent operator with a chaotic phase, i.e. a cluster having an a.c. invariant measure, and a cluster with an attracting fixed point.

\begin{example} \label{ex:chimera2}
Divide the interval $[0,1]$ (thought of having the extrema identified, $0\sim 1$) into two intervals: $I_1=[0,1/2]$ and $I_2=[1/2,1]$. Consider $f:[0,1]\rightarrow [0,1]$ defined as
\[
f|_{I_1}(x)= 4x(1-2x),
\] 
i.e. $f$  is a rescaled version of the logistic map $4x(x-1)$ on $I_1$; $f|_{I_2}$ joins smoothly with $f|_{I_1}$, has a single repelling fixed point $x^*$ (aside from $1\sim 0$) and is defined such that full Lebesgue measure of trajectories leave $I_2$ eventually (see Figure \ref{fig:chimera2}). Notice that by construction $f(I_1)\subset I_1$, and furthermore there is a unique a.c.p. measure invariant under $f$ is $\eta$ with density
\[
\psi(x)=\left\{\begin{array}{ll}\frac{2}{\pi\sqrt{2x(1-2x)}} & x\in I_1\\
0 & x\in I_2
\end{array}
\right.
\] 
Consider $h(x,y)= \alpha v(x)u(y)$ with $u,v:[0,1]\rightarrow \R$ periodic and smooth such that $v(x)=0$ on $I_1$ and at $x^*$, while $v'(x^*)=-1$. $u$ is a positive function on $I_1$ and zero on $I_2$, and finally $\alpha$ is a real parameter. With these prescriptions equation \eqref{Eq:ChimeraSplitting} becomes 
\begin{align*}
F_{(\nu_1,\nu_2)}(x)=f\left(x+\alpha v(x)\ell \int ud\nu_1\right).
\end{align*}
First of all, notice that $( \nu_1,\nu_2)=(\eta,\delta_{x^*})$ is a fixed state under the self-consistent operator (since $v$ is zero on $I_1$ and at $x^*$, so $F_{(\nu_1,\nu_2)}$ equals $f$ on $I_1$ and at $x^*$.)   Since $u$ is positive on $I_1$, there is $K>0$ such that 
\[
K:=\int ud \nu_1>0.
\]
Now we tune $\alpha\ell$ in such a way that  $|1-\alpha\ell K|<1/|f'(x^*)|$. With this choice, for every $\nu_2$, $|F_{(\nu_1 ,\nu_2)}'(x^*)|=|f'(x^*)(1-\alpha\ell K)|<1$, and therefore $|F_{( \nu_1,\nu_2)}'(x)|\le \lambda<1$ for every $x$ in some neighborhood $U$ of $x^*$.  We can apply Theorem \ref{Thm:PartialSynchronization} with this set $U$ and $\mc N'=\mc N=\{\eta\}$ and obtain that $d_W(\mc F^n (\eta,\nu_2),\mc {PS}_{U}) \to 0$.

In fact more is true. If we take $\nu_2$ supported on $U$, then under application of the self-consistent operator the state of cluster 2 converges to $\delta_{x^*}$. Using the properties of the logistic map, one can show that starting with $\nu_1$ a suitable (small) perturbation of $\eta$ (supported on $I_1$), $\mc F^n(\nu_1,\nu_2)$ converges to $(\eta,\delta_{x^*})$.

\end{example}

\subsection{Numerical evidence of chimeras in finite networks}
Below we present some simulations showing how the chimera states, $(\mu_1,\mu_2)$, found in the examples above for the self-consistent transfer operator can be numerically detected also in the corresponding systems of finite size. We would like to stress that the simulations presented in this section have mostly illustrative purposes.  

We start from a system of $N$ coupled units evolving as in Eq. \eqref{Eq:CoupledSystFullSymm}. Assuming $N$ is even, we divide the units into two clusters of size $N/2$ each. We draw an initial condition $\bo \xi_0=(\xi_{0,1},...,\xi_{0,N})$ in the following way: for $1\le i\le N/2$ we draw $\xi_{0,i}$ at random according to a probability measure close to $\mu_1$, while for $N/2+1\le i\le N$, we draw $\xi_{0,i}$ according to a probability measure close to $\mu_2$. 

We then let the initial condition evolve   according to the set of $N$ discrete equations in Eq. \eqref{Eq:CoupledSystFullSymm} and thus get a piece of orbit $\{\bo \xi_t\}_{t=1}^T$ for some  $T>0$. For a few  values of time $t$, we plot  the histograms for the points $\{\xi_{t,i}\}_{i=1}^{N/2}$ in cluster 1, and $\{\xi_{t,i}\}_{i=N/2+1}^N$ in cluster 2 to get a visual of the distribution of these points.   Then, for every $t=1,..,T$, we compare the empirical distribution obtained from $\{\xi_{t,i}\}_{i=1}^{N/2}$ 
with that of $\mu_1$, and the empirical distribution of $\{\xi_{t,i}\}_{i=N/2+1}^N$ with that of $\mu_2$ by numerically computing
\[
\mc D_1(t)=d_W\left(\frac{1}{N/2}\sum_{i=1}^{N/2}\delta_{\xi_{t,i}},\mu_1\right)\quad\mbox {and}\quad \mc D_2(t)=d_W\left(\frac{1}{N/2}\sum_{i=N/2+1}^{N}\delta_{\xi_{t,i}},\mu_2\right).
\]
where $d_W$ denotes the Wasserstein distance\footnote{ In the current setup, we can compute the Wasserstein distance $d_W(\nu,\nu')$ as the $L^1$ norm of the pseudoinverses of the cumulative distribution functions of the probability measures $\nu$ and $\nu'$ \cite{}.}. We observe that $\mc D_1$ and $\mc D_2$ tend to remain small across the time span analyzed ($T=1500$).  Then we  study how $\mc D_1$ and $\mc D_2$ vary varying $N$. We expect these values to decrease as $N$ increases since for large $N$, the finite system should be better approximated by the self-consistent operator. To do so, we average $\mc D_i(t)$ for the values obtained when $1000\le t\le 1500$, and plot this average values as a function of $N$ with error bars denoting max and min of $\mc D_i(t)$ on the interval of $t$ considered.

The results of this analysis for Example \ref{ex:chimera} and Example \ref{ex:chimera2} are reported in Fig. \ref{Fig:Example1} and Fig. \ref{Fig:Example2}. 

We performed simulations for larger time spans that are in accord with what is observed for time spans showed in the figures below.

The simulations we present in this section have mainly illustrative purpose, and a more careful numerical analysis would be needed to draw any quantitative conclusion. 

\begin{figure}[h!]
\begin{center}
\begin{subfigure}{1\textwidth}
\begin{center}
\includegraphics[scale=0.45]{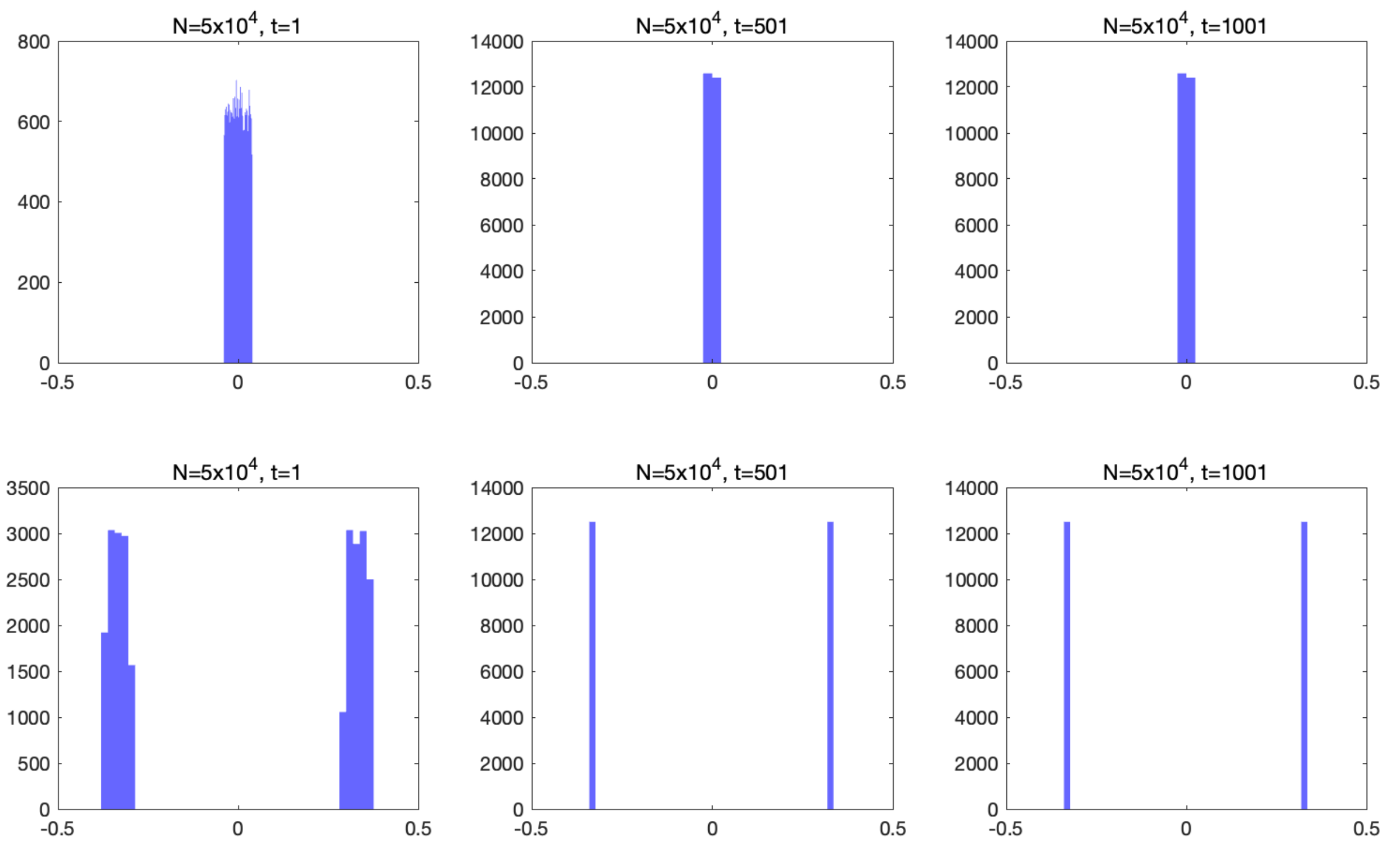}
\caption{}
\end{center}
\end{subfigure}
\begin{subfigure}{1\textwidth}
\begin{center}
\includegraphics[scale=0.45]{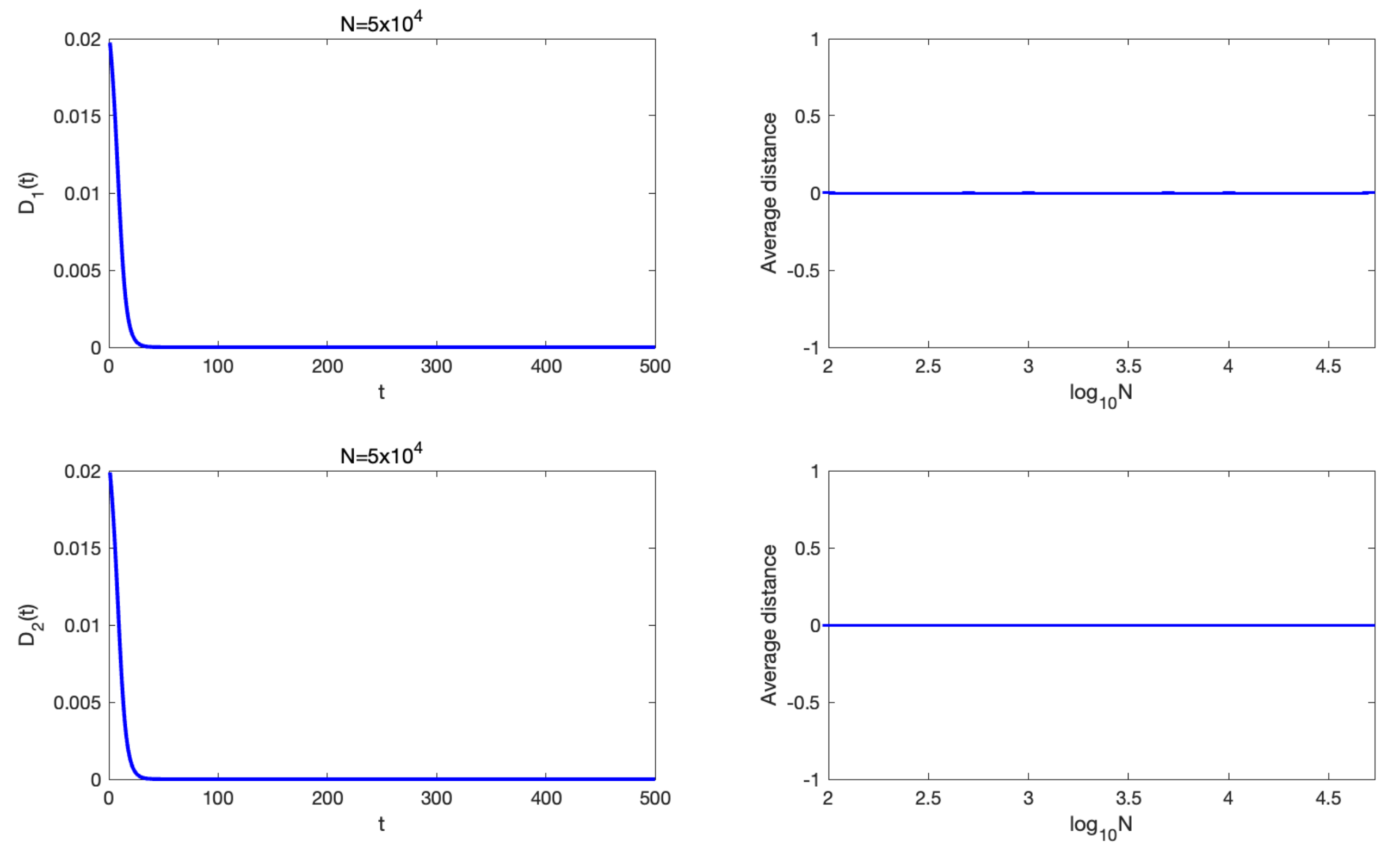}
\caption{}
\end{center}
\end{subfigure}
\end{center}
\caption{Result of the numerical analysis of the dynamic in \eqref{Eq:CoupledSystFullSymm} with $f(x)=2x$ mod 1 and $h(x,y)=-\frac{1}{10\pi}\sin(6\pi x)\cos(6\pi y)$ as in Example \ref{ex:chimera}. The first column in Panel A)  reports the histograms for initial conditions in cluster 1, first row,  and cluster 2, second row,  when $N=5*10^4$. Second and third column report the histograms at two later instants of time after 500 and 1000 time steps respectively. Notice that we represented $\T$ as the interval $[-1/2,1/2]$ with extrema identified. We observe that the points tend to pile up around 0 in cluster 1 and to evenly distribute around $-1/3$ and $1/3$ in cluster 2 (the fixed point at $2/3$ in the representation  of $\T$ as $[0,1]$, corresponds to $-1/3$ in the representation of $\T$ as $[-1/2,1/2]$). The first column of panel B)  shows $\mc D_1$ and $\mc D_2$ as a function of time when $N=5*10^4$. The last column on panel B) shows averages of $\mc D_i(t)$ varying $N=10^2,\,5*10^2,\,10^3,\,5*10^3,\,10^4,\,5*10^4$. When the average of $\mc D_i(t)$ is below $10^{-5}$ we draw a point at zero. }
\label{Fig:Example1}
\end{figure}

\begin{figure}[h!]
\begin{center}
\begin{subfigure}{1\textwidth}
\begin{center}
\includegraphics[scale=0.45]{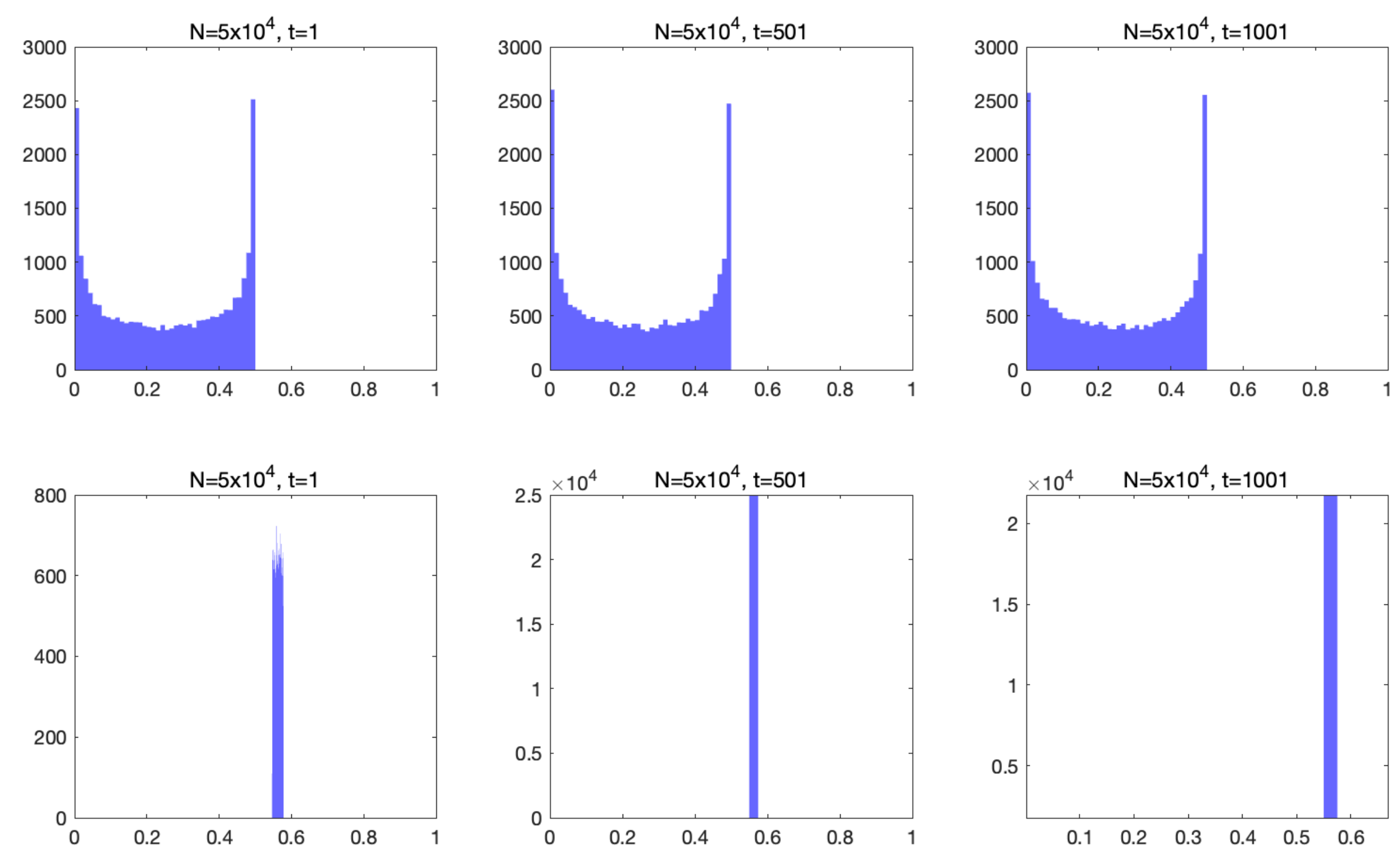}
\caption{}
\end{center}
\end{subfigure}
\begin{subfigure}{1\textwidth}
\begin{center}
\includegraphics[scale=0.45]{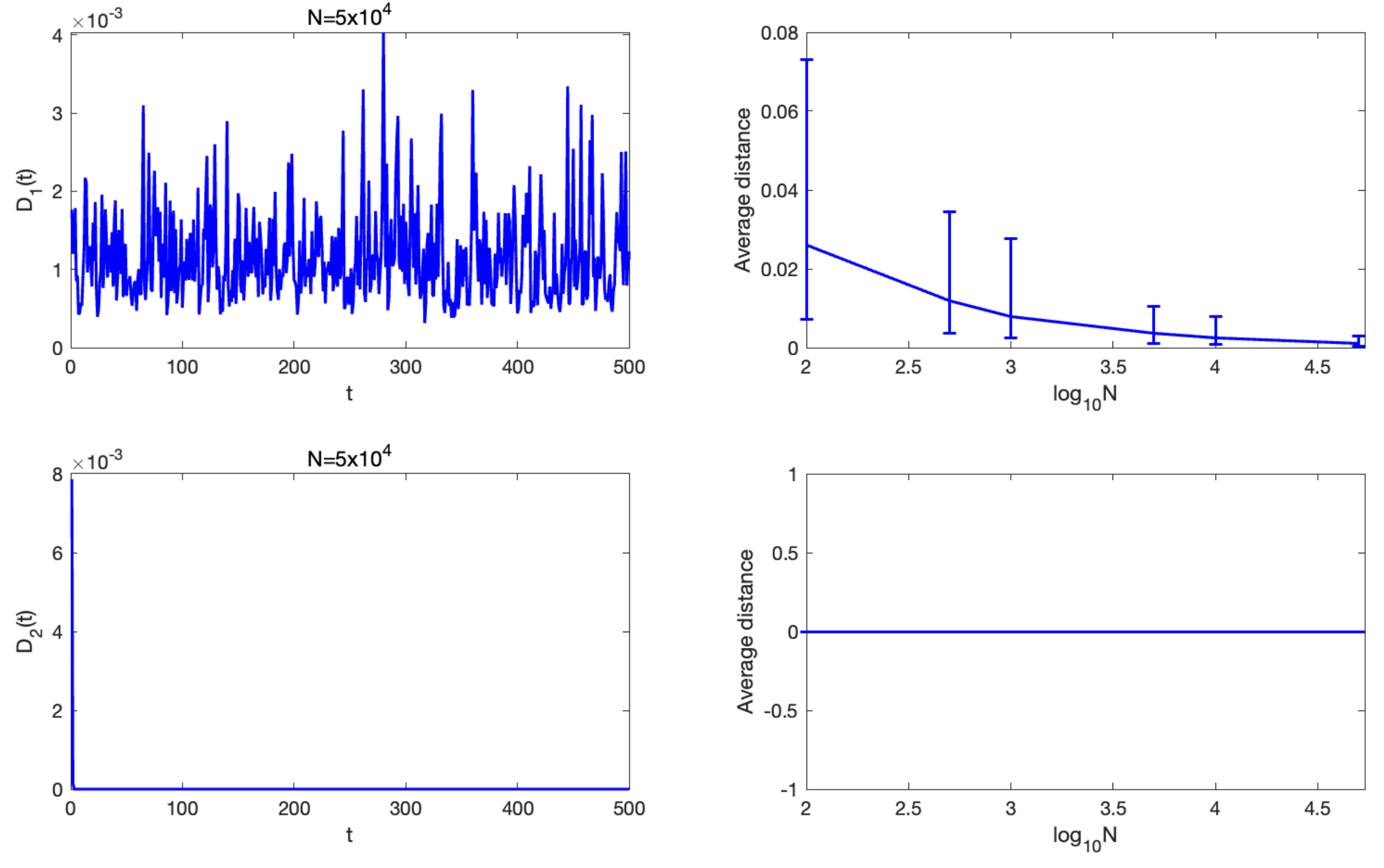}
\caption{}
\end{center}
\end{subfigure}
\end{center}
\caption{Result of the numerical analysis of the dynamic in \eqref{Eq:CoupledSystFullSymm} with $f(x)$ mod 1 and $h(x,y)$ as in Example \ref{ex:chimera2}. The first column in Panel A)  reports  the histograms for initial conditions in cluster 1, first row, and cluster 2, second row, respectively when $N=5*10^4$. Second and third column report the histograms at two later instants of time after 500 and 1000 time steps respectively. Here we represented $\T$ as the interval $[0,1]$ with extrema identified. We observe that the points tend to pile up around $x^*=0.5625$ in cluster 2 and to  distribute according to the density $\psi$ in cluster 1. The first column of panel B)  shows $\mc D_1$ and $\mc D_2$ as a function of time when $N=5*10^4$. The last column shows averages of $\mc D_i(t)$ varying $N=10^2,\,5*10^2,\,10^3,\,5*10^3,\,10^4,\,5*10^4$. When the average of $\mc D_i(t)$ is below $10^{-5}$ we draw a point at zero. The last column on panel B)  we distinctively observe that the values obtained for $\mc D_1$ decrease increasing $N$.}
\label{Fig:Example2}
\end{figure}

\FloatBarrier
\bibliographystyle{alpha}
\bibliography{Bibliography}

\end{document}